\documentclass[11pt]{article}

\usepackage{amsfonts, amsmath, mathtools, amssymb, amsthm, xcolor, enumitem, dsfont, comment}

\usepackage[numbers]{natbib}
\usepackage[pdfauthor = {DM}, colorlinks = true, citecolor = blue]{hyperref} 
\usepackage[english]{babel}

\newtheorem{theorem}{Theorem}[section]

\newtheorem{lemma}{Lemma}[section]
\newtheorem{proposition}{Proposition}[section]
\newtheorem{cor}{Corollary}[section]
\newtheorem{assumption}{Assumption}
\newtheorem{assumptionprime}{Assumption}

\theoremstyle{remark}
\newtheorem{remark}{Remark}[section]

\theoremstyle{definition}

\newcommand{\E}{\mathbb{E}} 
\renewcommand{\P}{\mathbb{P}}
\newcommand{\Var}{\operatorname{Var}} 
\newcommand{\R}{{\mathbb R}}
\newcommand{\N}{{\mathbb N}}
\newcommand{\dd}{\mathrm{d}}
\newcommand{\ii}{\mathfrak{i}}
\newcommand{\bone}{\mathds 1}

\DeclareMathOperator{\Leb}{Leb}

\usepackage{authblk}

\title{Marcinkiewicz--Zygmund-type SLLN for mixed moving average processes}

\author[1]{Danijel Grahovac\thanks{dgrahova@mathos.hr}}
\author[2]{P\'eter Kevei\thanks{kevei@math.u-szeged.hu}}
\author[3]{Dominik Mihal\v{c}i\'c \thanks{dmihalci@mathos.hr}}

\affil[1]{School of Applied Mathematics and Informatics, J. J. Strossmayer University of Osijek, Trg Ljudevita Gaja 6, 31000 Osijek, Croatia}
\affil[2]{Bolyai Institute, University of Szeged, Aradi v\'ertan\'uk tere 1, 6720 Szeged, Hungary}
\affil[3]{School of Applied Mathematics and Informatics, J. J. Strossmayer University of Osijek, Trg Ljudevita Gaja 6, 31000 Osijek, Croatia}

\date{}

\begin{document}

\maketitle

\begin{abstract}
The Marcinkiewicz--Zygmund theorem is a fundamental result in probability theory that establishes rates of convergence in the strong law of large numbers (SLLN). Although numerous extensions have been developed for dependent sequences, many classes of processes, particularly those exhibiting strong dependence, remain unexplored.
In this paper, we present a Marcinkiewicz--Zygmund-type SLLN for a class of mixed moving average processes, which form a large and flexible class of stationary infinitely
divisible processes. In contrast to the classical case, where moments determine the asymptotic behavior, the present setting additionally involves key objects that characterize both dependence and marginal distributions.
\medskip


\end{abstract}

\section{Introduction}

The celebrated Marcinkiewicz--Zygmund strong law of large numbers 
\cite{gut2013, Marcinkiewicz1937} is a refinement of the classical 
SLLN stating that, for a sequence of independent identically distributed
random variables $(Y_n, \, n \in \N)$, if $\E |Y_n|^\gamma<\infty$ 
for $\gamma \in (0,2)$, then
\begin{equation} \label{eq:MZ-classical}
	\frac{S_n-\bone(\gamma \geq 1) \E S_n}{n^{1/\gamma}} \to 0 \quad \text{a.s.},
\end{equation}
where $S_n = \sum_{i = 1}^n Y_i$. There is an abundance of papers extending
this result to sequences of dependent random variables. In particular, it was
studied under various mixing conditions in \cite{berbee1987, chandra1996, rio1995, shao1995}.
Further progress was made in \cite{louhichi2000}, where the SLLN for linear processes was obtained,
and in \cite{fazekas2001}, which provided a general approach for analyzing the 
rate of convergence in the SLLN. A well-known SLLN for $\R$-valued L\'evy processes was 
recently extended to Markov processes in \cite{kyprianou2025, kyprianou2020,  yaran2024}. 
Recent contributions to this area also include \cite{anh2021, behme2025, benhariz2025}.

In this paper, we prove a Marcinkiewicz--Zygmund-type SLLN for infinitely 
divisible mixed moving average (MMA) processes. MMA processes are defined by the stochastic integral
\begin{equation} \label{eq:MMA-definition-intro}
	X(t) = \int_{V \times \R} f(x, t-s) \Lambda(\dd x, \dd s), \quad t \geq 0,
\end{equation}
where $(V, \mathcal{V}, \pi )$ is a $\sigma$-finite measure space, 
$f\colon V\times \R \to [0,\infty)$ is a measurable function and 
$\Lambda$ is an infinitely divisible independently scattered random measure determined by
the L\'evy--Khintchine triplet $(a,b,\lambda)$ and control measure $C_\Lambda(\pi \times \Leb)$ 
(see Section~\ref{sec:preliminaries} and \cite{barndorff2018book, samorodnitsky2016} for details).  
MMA processes form a large and flexible class of stationary, infinitely divisible processes. 
Among them are the superpositions of Ornstein--Uhlenbeck (supOU) processes introduced by 
Barndorff-Nielsen in \cite{barndorff2000supOU}, for which $V=(0,\infty)$ and the integrand 
in \eqref{eq:MMA-definition-intro} is 
\[
f(x, v) = e^{-x v} \bone(v \geq 0), \quad  x > 0.
\]
Another important class of MMA processes is trawl processes
\cite{barndorff2011statID,barndorff2014intTrawl}, obtained by taking 
\[
f(x, v) = \bone(v \geq 0, \, 0 \leq x \leq \psi(v)), \quad x \in (0, \psi(0)],
\]
for a suitable function $\psi$, $V = (0, \psi(0)]$ and $\pi  \equiv \Leb$ (see Section~\ref{sec:examples} for details).
In particular, if $V$ is a one-point set, \eqref{eq:MMA-definition-intro} reduces to a L\'evy driven moving average (MA) process
\begin{equation} \label{eq:MA-definition}
	X(t) = \int_{\R} f(t-s) L(\dd s), \quad t \in \R,
\end{equation}
where $L$ is a two-sided L\'evy process with L\'evy--Khintchine triplet $(a,b,\lambda )$. 
Asymptotic behavior of MA processes has been investigated in numerous works; see, for example, 
\cite{basseOconnor2019functionals, basseOconnor2016power, basseOconnor2016semimart,  benassi2004roughness,bender2012finite,marquardt2006fractional, ronnielsen2022extremes} among many others. 

{A Marcinkiewicz--Zygmund-type SLLN for supOU processes was obtained in
	\cite{grahovac2025}, and our main results are far-reaching generalizations of
	Theorems 1 and 2 there. }
Since we are dealing with continuous-time processes, a natural analogue of the partial sums
from \eqref{eq:MZ-classical} is the time integrated process $X^* = \{X^*(t), \, t \geq 0\}$ given by
\begin{equation} \label{eq:integrated-process}
	X^*(t) = \int_0^t X(u) \dd u.
\end{equation}
We show that the role of the moment condition in Marcinkiewicz--Zygmund SLLN 
\eqref{eq:MZ-classical} is played by the following condition
\begin{equation} \label{eq:Cgamma-condition-intro} 
	\int_V \int_{\R} |z|^{\gamma} f_1(x)^{\gamma} \bone (|z|f_1(x) > 1) 
	\lambda (\dd z) \pi (\dd x) < \infty,
\end{equation}
where $f_1(x) = \int_0^\infty f(x,v) \dd v$. Provided that some additional assumptions hold
for the kernel function $f$ (see Section~\ref{sec:preliminaries} for details), we show that in the typical case,
if \eqref{eq:Cgamma-condition-intro} holds with $\gamma \in (0,2)$, then
\begin{equation*}
	\lim_{t \to \infty} \frac{X^*(t) - \bone(\gamma \geq 1) \E X^*(t)}{t^{1/\gamma}} = 0 \quad \text{a.s.}	
\end{equation*}
{Additionally, if \eqref{eq:Cgamma-condition-intro} holds with $\gamma = 2$, then we have the law of the iterated logarithm
\begin{equation*}
	\limsup_{t \to \infty} \frac{|X^*(t) - \E X^*(t)|}{\sigma \sqrt{2 t \log \log t}} = 1 \quad \text{a.s.},
\end{equation*}
where $\sigma^2$ is the asymptotic variance of the integrated process, 
that is $\sigma^2=\lim_{t\to\infty}t^{-1}\Var X^*(t)$.}
Note that if $V$ is a one point set, then \eqref{eq:Cgamma-condition-intro} 
reduces to $\int_{\R} |z|^\gamma \bone (|z| > 1) \lambda (\dd z) < \infty$, 
which is equivalent to $\E |L(1)|^\gamma < \infty$ for $L$ being 
the driving L\'evy process in \eqref{eq:MA-definition} (see \cite{sato2013}). 
In general, however, \eqref{eq:Cgamma-condition-intro} depends not only on the tail behavior of
the L\'evy measure $\lambda $, but also on its behavior at zero and on the behavior of
$\pi$ and $f_1$ both near zero and at infinity. 

The main message of our results is that the structure of the Marcinkiewicz--Zygmund SLLN survives under strong dependence. Condition
\eqref{eq:Cgamma-condition-intro} plays the role of the moment condition in
\eqref{eq:MZ-classical}. The kernel enters it only through $f_1(x)$, the
total contribution of a single Poisson point, while $\lambda$ and $\pi$
determine the distribution of these contributions. The temporal shape of
the kernel, which governs the dependence, is controlled by two conditions
stated in Section~\ref{subsec:assumptions}. These conditions concern $f$
alone, do not involve $\gamma$, $\lambda$ or $\pi$, and cover most of the
known examples of MMA processes as well as L\'evy driven moving averages.
Once they hold, the growth rate of $X^*$ is determined by
\eqref{eq:Cgamma-condition-intro} in the same way as it is determined by moments in the
classical case. In the supOU case of \cite{grahovac2025}, the kernel is
explicit and its exponential form is used at every step, so the role of
the kernel could not be separated from that of $\lambda$ and $\pi$.

Our results provide new insight into the almost sure behavior of processes under strong dependence.
Results of this type are scarce, with only a few known contributions, including 
\cite{kouritzin2022, louhichi2000} and \cite{fuchs2013mixing}, where the mixing property was 
shown for MMA processes. Moreover, we precisely quantify how the almost sure rate of growth depends on
the parameters governing the behavior of $\lambda $, $\pi $, and $f_1$ at both zero and infinity.
The results apply to a broad class of processes, and the examples we present demonstrate that our
assumptions are straightforward to verify in concrete cases.

{A distributional counterpart of these questions, namely the extremal behavior of MMA
processes, was studied in \cite{fasen2005rv, fasen2009conv}. There it is assumed that $f$
is bounded and dominated by a non-increasing integrable function, while $\pi$ is a probability measure, and $\lambda$ is
either regularly varying or convolution equivalent. Our assumptions allow unbounded kernels and an
infinite measure $\pi$, while imposing no structural condition on $\lambda$ beyond the integrability required by
\eqref{eq:Cgamma-condition-intro}.}

MMA processes have the general form of an aggregation of simple processes. 
For instance, if $\pi$ is a measure with finite support, then the corresponding MMA process
reduces to a finite sum of independent MA processes. Such aggregated models have found numerous
applications in economics and finance \cite{barndorff2013volatility, fuchs2013mixing, granger1980}.

The paper is organized as follows. Section~\ref{sec:preliminaries} introduces the main tools and preliminaries,
and also lists the general assumptions together with a discussion of their restrictiveness. 
The main results are presented in Section~\ref{sec-main}. Section~\ref{sec:examples} provides a variety of examples
illustrating the applicability of the main theorems. All proofs are collected in Section~\ref{sec:proofs}.

\section{Preliminaries and assumptions} \label{sec:preliminaries}

The Lebesgue measure on $\R$
is denoted by $\Leb$, while $\bone(A)$ denotes the indicator of a set $A$.
For $a_1,a_2\in\R$, we write $a_1\wedge a_2=\min\{a_1,a_2\}$. {For a measure $Q$ on $(0,\infty)$
and $p \in \R$, we set 
\[
m_p(Q) \coloneqq \int_{(0,\infty)} x^p Q(\dd x),
\]
and denote the tail of $Q$ by $\overline{Q}(x) \coloneqq Q((x,\infty))$.}

\subsection{Infinitely divisible random measures and integrals}

{Let $(V,\mathcal{V})$ be a Borel space (see e.g.~\cite{kallenberg2021foundations}), that is, a measurable space
that is Borel isomorphic to a Borel subset of $\R$, equipped with a
$\sigma$-finite measure $\pi$. We consider the product $V \times \R$ equipped
with the product $\sigma$-algebra $\mathcal{V} \otimes \mathcal{B}(\R)$, and denote
$\mathcal{B}_0(V \times \R) = \{A \in \mathcal{V} \otimes \mathcal{B}(\R)
\colon (\pi \times \Leb)(A) < \infty \}$.}

Let $\Lambda$ be a homogeneous infinitely divisible independently scattered random
measure on $V \times \R$ such that for all $A \in \mathcal{B}_0(V \times \R)$ and
$\theta \in \R$ 
\begin{equation} \label{eq:levy-basis-cf}
\begin{split}
& \log \E e^{\ii \theta \Lambda(A)} \\
& = 
	(\pi  \times \Leb) (A) \left(
	\ii \theta a - 
	\frac{\theta^2}{2} b  + 
	\int_{\R} \left(e^{\ii \theta z} - 1 - \ii\theta z \bone (|z| \leq 1) \right)
	\lambda (\dd z)\right), 
\end{split}
\end{equation}
where $a \in \R$, $b \geq 0$ and $\lambda$ is a L\'evy measure on $\R$. 
Such a random measure is called a L\'evy basis and the quadruple $(a,b,\lambda,\pi)
$ will be referred to as its generating quadruple. Moreover, the product measure 
$C_\Lambda(\pi  \times \Leb)$ is the control measure of $\Lambda$, with 
$C_\Lambda = |a| + b + \int (1 \wedge z^2) \lambda(\dd z)$; see \cite[Section 5.1.3]{barndorff2018book}. 

The integral of a deterministic function with respect to $\Lambda$ is first defined
for simple functions. For a general measurable function, the integral is defined as
the limit in probability of the integrals of an approximating sequence of simple functions
(see \cite{rajput1989} or \cite[Chapter 5]{barndorff2018book} for details). 
The necessary  and sufficient conditions for the $\Lambda$-integrability of some measurable function $f$ were established in \cite[Theorem 2.7]{rajput1989}
(see also \cite[Proposition 34]{barndorff2018book}).
In our framework, \cite[Proposition 34]{barndorff2018book} takes the form
\begin{equation} \label{eq:RR-cond}
\begin{split}
&  \int_V \int_{\R} |U(f(x, v))| \dd v \pi (\dd x) < \infty, \\
&  \int_V \int_{\R} b f(x, v)^2 \dd v \pi (\dd x) < \infty, \\
&  \int_V \int_{\R} V_0(f(x, v)) \dd v \pi (\dd x) < \infty,
\end{split}
\end{equation}
where 
\begin{equation} \label{eq:UV0}
\begin{split}
	U(y) & =  a y + y \int_{\R} z \left( \bone(|y z| \leq 1) - \bone(|z| \leq 1) \right) \lambda (\dd z), \\ 
	V_0(y) & = \int_{\R} \left( 1 \wedge y^2 z^2 \right) \lambda (\dd z). 
\end{split}
\end{equation}

Since the control measure $C_\Lambda(\pi  \times \Leb)$ of the random measure $\Lambda$ is diffuse, i.e., has no atoms, 
there exists a modification of $\Lambda$ that admits a L\'evy--It\^o decomposition as 
in \cite[Theorem 4.5]{pedersen2003} (see also \cite[Proposition 31]{barndorff2018book}). 
Therefore, 
\begin{equation} \label{eq:levy-ito} 
\begin{split}
    \Lambda(\dd x, \dd s) & = 
	a \, \pi(\dd x) \dd s + 
	\Lambda^G(\dd x, \dd s) \\
    & \quad + 
	\int_{|z| \leq 1 } z (\mu-\nu) (\dd x,\dd s,\dd z) + 
	\int_{|z|>1} z \mu(\dd x,\dd s,\dd z),
\end{split}
\end{equation}
where $\Lambda^G$ is a Gaussian random measure with generating quadruple $(0, b, 0, \pi )$
and $\mu$ is a Poisson random measure on $V \times \R \times \R$ with 
intensity measure $\nu(\dd x,\dd s,\dd z) = \pi(\dd x) \dd s \lambda(\dd z)$. 
{We denote by $(\xi_k, \tau_k, \zeta_k)_{k \geq 0}$ the points of the Poisson random measure $\mu$.}

\subsection{Mixed moving average processes} \label{subsec:mma}

Let $\Lambda$ be an infinitely divisible random measure on $V \times \R$ as in \eqref{eq:levy-basis-cf},
and let $f \colon V \times \R \to [0,\infty)$ be a $\Lambda$-integrable function such that 
$f \in L^1(\pi  \times \Leb)$. A mixed moving average (MMA) process $X = \{X(t), \, t \in \R\}$ is defined by
\begin{equation}\label{eq:MMA-definition}
	X(t) = \int_{V \times \R} f(x, t-s) \Lambda(\dd x,\dd s).
\end{equation}
Clearly, $X$ is a strictly stationary process. We restrict our attention to 
causal processes by assuming $f(x,v) = 0$ for $v < 0$. 

The cumulative process in the context of Marcinkiewicz--Zygmund SLLN is the time 
integrated process \eqref{eq:integrated-process}. The process $X^*$ has stationary increments
and it falls in the class of stationary increments mixed moving average 
(SIMMA) processes \cite{basseOconnor2013finiteVar}. 

We introduce the functions
\begin{equation*}
	f_1(x) \coloneqq \int_0^\infty f(x,v) \dd v, \quad 
	f_2(x,u) \coloneqq \int_u^\infty f(x,v) \dd v, \quad 
	x \in V, \, u \geq 0.
\end{equation*} 
{Note that $f \geq 0$ ensures that $f_2(x,\cdot)$ is non-increasing and that
	$f_2(x,u) \leq f_1(x)$ for all $x \in V$ and $u \geq 0$. These two properties
	are what the assumptions and the proofs below rely on, and they concern $f$
	only through its second argument. For general kernels, one can apply our results to positive and negative parts separately.}

Provided that we can interchange the order of integration between the stochastic integral
defining $X$ and the time integral in \eqref{eq:integrated-process}, the integrated process
can be represented as follows
\begin{equation} \label{eq:decomposition-X*}
	\begin{split}
		X^*(t) &= 
		\int_{V \times \R} 
		\left(f_2(x, -s) - f_2(x, t-s) \right) \bone(s \leq 0)
		\Lambda(\dd x, \dd s) \\ 
		& \quad + \int_{V \times \R}
		\left( f_1(x) - f_2(x, t-s) \right) \bone(0 < s \leq t)	 		
		\Lambda(\dd x, \dd s).
	\end{split}
\end{equation}
The interchange of the order of integration can be justified by the stochastic
Fubini theorem \cite[Theorem 3.1]{barndorff2011qou}. This theorem is applicable when 
$\E|X(1)|<\infty$. If $\E|X(1)|=\infty$, we will assume the following
\begin{align}
	\int_V \int_\R \int_{|z|\leq 1} |z| f(x,s) \bone(|z| f(x,s) > 1)
	\lambda(\dd z) \dd s \pi(\dd x) < \infty, \label{eq:fubini-cond-1} \\
	\int_V \int_{|z|>1} (1 \wedge |z| f_1(x) ) \bone(f_1(x) \leq 1)
	\lambda(\dd z) \pi(\dd x) < \infty, \label{eq:fubini-cond-2}
\end{align}
and the existence of a function $g \colon V \times \R \to [0,\infty)$ such that for
every $t>0$, $x \in V$ and $s\leq 0$
\begin{equation} \label{eq:fubini-cond-3}
\begin{split}
& \int_0^t f(x,u-s) \dd u \leq t g(x,-s) \quad \text{and} \\
& \int_V \int_0^{\infty} \int_{|z|>1}
\left(1 \wedge |z| g(x,s) \right) \lambda(\dd z) \dd s \pi(\dd x) < \infty.
\end{split}
\end{equation}

\begin{lemma} \label{lemma:StochFubini} 
	If either $\E |X(1)| < \infty$, or \eqref{eq:fubini-cond-1}, \eqref{eq:fubini-cond-2} and \eqref{eq:fubini-cond-3} hold, then 
	\begin{equation*}
		X^*(t) = \int_{V \times \R} \int_0^t f(x, u-s) \dd u \Lambda(\dd x, \dd s).
	\end{equation*}	
\end{lemma}

We thus
assume \eqref{eq:fubini-cond-1}--\eqref{eq:fubini-cond-3} to hold whenever
$\E|X(1)|=\infty$.

\subsection{Assumptions} \label{subsec:assumptions}

{We first gather all standing assumptions for reference.
	\begin{assumption} \label{assum-0}
		For an MMA process \eqref{eq:MMA-definition} we assume $f$ is a non-negative $\Lambda$-integrable function such that $f \in L^1(\pi  \times \Leb)$ and $f(x,v) = 0$ for $v < 0$. If $\E|X(1)|=\infty$, we assume additionally \eqref{eq:fubini-cond-1}--\eqref{eq:fubini-cond-3} to hold.
	\end{assumption}
}

{
The main results rely on two conditions, both concerning the kernel $f$ alone. 
\begin{assumption} \label{assum-1}
	There exists $K>0$ such that for all $x \in V$
	\begin{equation*}
		\int_0^\infty s f(x,s)\dd s
		\leq K f_1(x)^2.
	\end{equation*}
\end{assumption}

Assumption~\ref{assum-1} admits a natural interpretation in terms of the
persistence of the contribution of individual Poisson points. A point
$(\xi_k,\tau_k,\zeta_k)$ of $\mu$ contributes $\zeta_k f(\xi_k,v)$ to the
process at time $\tau_k+v$, so that $\zeta_k f_1(\xi_k)$ is its total
contribution and $f_2(\xi_k,r)/f_1(\xi_k)$ the proportion of it that remains
to be realised after the lag $r$. Since
$f_2(\xi_k,r) \leq r^{-1}\int_0^\infty u f(\xi_k,u)\dd u \leq K f_1(\xi_k)^2/r$,
this proportion is at most $K/s$ when $r = s f_1(\xi_k)$. Thus, after a lag of
$s$ units of the time scale $f_1(\xi_k)$, at most a proportion $K/s$ of the
contribution of the point remains.
}

The second assumption is related to the moments of $X^*$. In particular, it will be used
together with \cite[Theorem 1]{marinelli2014} to obtain appropriate bounds for the 
moments of components of $X^*$ given in \eqref{eq:decomposition-X*}. 
{These bounds in turn yield the path regularity of $X^*$, which allows us to reduce
the almost sure statements to integer times, without assuming any regularity of $f$ itself.}
\begin{assumption} \label{assum-2}
	For each $p \in [1,2]$, there exists $c_p > 0$, such that for all $x \in V$ and $t > 0$ 
	\begin{align}
		\int_0^\infty (f_2(x, u) - f_2(x, t+u))^p \dd u 
		&\leq c_p \, ( t \wedge f_1(x) )^p \, f_1(x) \label{eq:f-cond-int1}\\
		\int_0^t (f_1(x) - f_2(x, u) )^p \dd u 
		&\leq c_p \, ( t \wedge f_1(x) )^p \, t. \label{eq:f-cond-int2}
	\end{align}
\end{assumption}
Assumption~\ref{assum-2} is closely related to Assumption~\ref{assum-1}. 
It implies Assumption~\ref{assum-1} in general, and the converse holds whenever $f$ is bounded.

\begin{lemma} \label{lemma:equiv-A1-A2}
	If Assumption~\ref{assum-1} holds and $f(x,s)\leq M$ for some $M>1$ and all
	$x\in V$, $s\geq0$, then Assumption~\ref{assum-2} holds. Conversely, if
	\eqref{eq:f-cond-int1} holds with $p=1$, then Assumption~\ref{assum-1} holds.
\end{lemma}

Condition \eqref{eq:f-cond-int1} admits a similar interpretation. For a point
$(\xi_k,\tau_k,\zeta_k)$ with $\tau_k\leq0$, the difference
$f_2(\xi_k,-\tau_k) - f_2(\xi_k,t-\tau_k)$, divided by $f_1(\xi_k)$, is the
proportion of the contribution of the point that is delivered within the
window $[0,t]$. Writing $\tau_k = w f_1(\xi_k)$ and $t = s f_1(\xi_k)$,
condition \eqref{eq:f-cond-int1} bounds the $L^p$ norm of this proportion by
$c_p^{1/p}(s \wedge 1)$, so that a window of length $s$ receives at most a
proportion of order $s \wedge 1$.

Condition \eqref{eq:f-cond-int2} is the analogue for points arriving inside
the window, that is with $\tau_k \in (0,t]$, for which the proportion
delivered is $\frac{f_1(\xi_k)-f_2(\xi_k,t-\tau_k)}{f_1(\xi_k)}$. In the
same units its $L^p$ norm is bounded by $c_p^{1/p}(s \wedge 1)s^{1/p}$, the
extra factor reflecting that the arrival time now ranges over an interval of
length $s$ rather than a half-line.

\section{Main results} \label{sec-main}

The main results are split into three cases. As already note, our main 
condition is that
\begin{equation} \label{eq:Cgamma-condition} \tag{$C_\gamma$}
	\int_V \int_{\R} |z|^{\gamma} 
	f_1(x)^{\gamma} \bone (|z|f_1(x) > 1) 
	\lambda(\dd z) \pi(\dd x) < \infty,
\end{equation}
{holds for some $\gamma \in (0,2]$.} We discuss the implications of this condition in Subsection~\ref{subsec:growth}.

\subsection{Integrable small jumps}

We assume that $\int_{|z| \leq 1} |z| \lambda (\dd z) < \infty$. Without loss of generality, 
we set {$a = \int_{|z|\leq 1} z \lambda (\dd z)$}, so that the compensating term cancels out. 
By \eqref{eq:levy-ito}, the L\'evy basis $\Lambda$ can be expressed as
\begin{equation} \label{eq:Lambda-finite-variation}
	\Lambda(\dd x, \dd s) = \int_{\R} z \, \mu(\dd x, \dd s, \dd z).
\end{equation}

\begin{theorem}\label{thm:finite-variation}
	{Let $X$ be an MMA process satisfying Assumption~\ref{assum-0} and let $X^*$ denote the integrated process.
	Furthermore, assume that {$a = \int_{|z|\leq 1} z \lambda (\dd z)$}. Let $\gamma \in (0,2]$ and if $\gamma>1$ 
	suppose that Assumption~\ref{assum-1} holds.}
	
	If $\gamma < 2$ and \eqref{eq:Cgamma-condition} holds, then 
	\begin{equation} \label{eq:finite-variation-result-1}
		\lim_{t \to \infty} 
		\frac{X^*(t) - \bone (\gamma \geq 1) \E X^*(t)}{t^{1/\gamma}} 
		= 0 \quad \text{a.s.}	
	\end{equation}
	Moreover, if \eqref{eq:Cgamma-condition} holds for $\gamma = 2$, then 
	\begin{equation} \label{eq:finite-variation-result-2}
		\limsup_{t \to \infty} 
		\frac{|X^*(t) - \E X^*(t)|} {\sqrt{2 t \log \log t}}
		= \left(\int_{\R} z^2 \lambda(\dd z) \int_V f_1(x)^2\pi(\dd x) \right)^{1/2} 
		\quad \text{a.s.}
	\end{equation}
\end{theorem}

\subsection{Non-integrable small jumps}

Next, we consider the case when $\int_{|z| \leq 1} |z| \lambda (\dd z) = \infty$ and $a = 0$. 
The compensating term is required in \eqref{eq:levy-ito}, giving  
\begin{equation*}
	\Lambda(\dd x, \dd s) = 
	\int_{|z| \leq 1} z (\mu - \nu ) (\dd x, \dd s, \dd z) + 
	\int_{|z| > 1} z \mu(\dd x, \dd s, \dd z).
\end{equation*}
To describe the behavior of $\pi$ at infinity, we define
\begin{equation*}
	\alpha_0 = \sup \left\{
	\alpha' \geq 0 \colon \int_V 
	f_1(x)^{1 + \alpha'} \bone(f_1(x) > 1) \pi (\dd x) < \infty
	\right\} \leq \infty.
\end{equation*}
Note that the set above is non-empty since $f \in L^1(\pi  \times \Leb)$, and therefore
$f_1 \in L^1(\pi)$. If $\int_V f_1(x)^{1 + \alpha_0} \bone(f_1(x) > 1) \pi(\dd x) < \infty$,
we set $\alpha = \alpha_0$, and otherwise we take $\alpha < \alpha_0$ arbitrarily close. 
Moreover, by the definition of $\alpha$ and the integrability of $f_1$, it is clear that 
$f_1 \in L^{1+\alpha}(\pi)$. 
Next, we define the parameters
\begin{align*}
	\beta_0 &= \inf \left\{
	\beta' \geq 0 \colon \int_{|z| \leq 1} |z|^{\beta'} \lambda (\dd z) < \infty
	\right\} \in [0,2],\\
	\eta_{\infty} &= \sup \left\{
	\eta' \geq 0 \colon \int_{|z| > 1} |z|^{\eta'} \lambda (\dd z) < \infty
	\right\} \leq \infty,
\end{align*}
which describe the behavior of the L\'evy measure $\lambda$ at zero and at infinity, respectively.
The parameter $\beta_0$ is the Blumenthal--Getoor index \cite{blumenthal-getoor} of the 
L\'evy measure $\lambda$. If $\int_{|z| \leq 1} |z| \lambda(\dd z) = \infty$, then $\beta_0 \geq 1$,
while the fact that $\lambda$ is a L\'evy measure ensures that $\beta_0 \leq 2$. 
The parameter $\eta_\infty$ is referred to as the tail index of the L\'evy measure $\lambda$. 
In what follows, if $\int_{|z| \leq 1} |z|^{\beta_0} \lambda(\dd z) < \infty$, we set $\beta = \beta_0$.
Otherwise, we choose $\beta > \beta_0$ arbitrarily close. Similarly, we take $\eta \leq \eta_{\infty}$,
with equality when the corresponding integral is finite. 

\begin{theorem} \label{thm:infinite-variation}
	{Let $X$ be an MMA process satisfying Assumption~\ref{assum-0} and let $X^*$ denote the integrated process.
	Furthermore, suppose that $\int_{|z| \leq 1} |z| \lambda (\dd z) = \infty$, $a = 0$,  
	and that Assumption~\ref{assum-2} is satisfied.}
	\begin{enumerate}[label = (\roman*)]
		\item If $\beta \leq 1 + \alpha$, $\gamma \leq 2$ and \eqref{eq:Cgamma-condition} holds, then
		\begin{equation} \label{eq:infinite-variation-result-1}
			\limsup_{t \to \infty} 
			\frac{|X^*(t) - \bone(\gamma \geq 1) \E X^*(t)|}{t^{1/\gamma} \log t} 
			\leq 1 \quad \text{a.s.}
		\end{equation}		
		\item If $\beta > 1 + \alpha$, $\gamma < (1-\alpha/\beta)^{-1}$ and
		\begin{equation*}
			\int_V \int_{|z| > 1}  
			|z|^{\gamma} f_1(x)^{\gamma} \bone(|z| f_1(x) > 1) 
			\lambda(\dd z) \pi(\dd x) < \infty, 
		\end{equation*}
		then
		\begin{equation*}
			\lim_{t \to \infty}
			\frac{X^*(t) - \bone(\gamma \geq 1) \E X^*(t)} {t^{1/\gamma}} 
			= 0 \quad \text{a.s.}
		\end{equation*}		
	\end{enumerate}	
\end{theorem}
Note that the integral in condition (ii) is taken for $|z| > 1$, in contrast 
to the integral in \eqref{eq:Cgamma-condition}.

\subsection{Gaussian case}

Here we consider the purely Gaussian case, that is $a = 0$, $b > 0$ and $\lambda  \equiv 0$. In this case, the law of iterated logarithm holds, which follows from a general result on Gaussian processes \cite[Theorem 1.1]{orey1972}. 

\begin{theorem} \label{tm-gaussian}
	Let $X^*$ be the integrated MMA process such that $a = 0$, $b > 0$ and $\lambda  \equiv 0$. If $\int_V f_1^2 \dd \pi < \infty$, then 
	\begin{equation*}
		\limsup_{t \to \infty} 
		\frac{|X^*(t)|}{\sqrt{2 b \int_V f_1^2  \dd \pi \,  t \log \log t }} = 
		1 \quad \text{a.s.}
	\end{equation*}
\end{theorem}

\subsection{Rate of growth} \label{subsec:growth}

The condition \eqref{eq:Cgamma-condition} is related to the behavior of 
$\pi$ and $\lambda$ at zero and infinity. If $\gamma \in (0, 2]$ and \eqref{eq:Cgamma-condition} holds, then
\begin{align*}
& \int_V f_1(x)^\gamma \bone(f_1(x) > 1) \pi (\dd x) 
\int_{|z|>1} |z|^\gamma \lambda(\dd z) \\
& \leq \int_V f_1(x)^\gamma \int_{(1/f_1(x), \infty)} |z|^\gamma \lambda(\dd z) 
\pi(\dd x) \\ 
& = \int_V \int_{\R} |z|^\gamma f_1(x)^\gamma \bone(|z|f_1(x)>1) \lambda(\dd z) \pi(\dd x) < \infty.	 
\end{align*}
Therefore, $\gamma \leq 1 + \alpha$ and $\gamma \leq \eta$. Moreover, if \eqref{eq:Cgamma-condition} holds
with any $\gamma > 0$, then it also holds with $\gamma = 0$, that is, 
\begin{equation} \label{eq:Cgamma-0}
	\int_{\R} \pi  \left(\{f_1(x) > |z|^{-1} \} \right) \lambda (\dd z) = 
	\int_V \lambda \left(\{ |z|f_1(x) > 1 \}\right) \pi (\dd x) < \infty.
\end{equation} 

If $\pi \left(\{ f_1 > r^{-1} \} \right) = h(r) r^{1+\alpha_0}$ around $0$ or 
$\lambda \left(\{ |z| > r \} \right) = h(r) r^{-\beta_0}$, where $h$ satisfies 
$\lim_{r \downarrow 0} h(r) r^\varepsilon = 0$ and $\lim_{r \downarrow 0} h(r)r^{-\varepsilon} = \infty$ 
for any $\varepsilon > 0$, then \eqref{eq:Cgamma-0} implies $\beta \leq 1 + \alpha$.

If $\pi \left(\{f_1(x) > |z|^{-1} \} \right) = h(z) z^{1+\alpha_0}$ around $0$ or
$\lambda \left(\{ |z|f_1(x) > 1 \}\right) = h(x) x^{-\beta_0}$, where $h$ satisfies 
$\lim_{x \downarrow 0} h(x) x^\varepsilon = 0$ and $\lim_{x \downarrow 0} h(x)x^{-\varepsilon} = \infty$ 
for any $\varepsilon > 0$, then \eqref{eq:Cgamma-0} implies $\beta \leq 1 + \alpha$.
However, it is possible to construct measures $\pi$ and $\lambda$ such that \eqref{eq:Cgamma-condition} holds
with some $\gamma > 0$ and $\beta > 1 + \alpha$; see the remark after Theorem 2 in \cite{grahovac2025}.

In general, the possible choice of $\gamma$ depends on $\alpha$, $\beta$, and $\eta$. 
Note that we can decompose the integral in \eqref{eq:Cgamma-condition} as 
\begin{equation} \label{eq:main-rates}
	\begin{split}
		& \int_V \int_{\R} |z|^\gamma f_1(x)^\gamma \bone(|z| f_1(x) > 1) \lambda (\dd z) \pi (\dd x)  \\ 
		&= \int_V f_1(x)^\gamma \bone(f_1(x) > 1) \pi (\dd x) 
		\left( \int_{|z| \in (1/f_1(x), 1]} 
		|z|^\gamma \lambda (\dd z) + \int_{|z| > 1} |z|^\gamma 
		\lambda (\dd z) \right) \\ 
		& \quad + \int_V f_1(x)^\gamma \bone(f_1(x) \leq 1) \pi (\dd x) 
		\int_{|z| > 1/f_1(x)} |z|^\gamma \lambda (\dd z).
	\end{split}	
\end{equation}
Assume that $\pi$ is a finite measure. By analyzing each of the terms in \eqref{eq:main-rates} separately,
we can deduce that the smallest possible value of $1/\gamma$ that can be taken in 
Theorems~\ref{thm:finite-variation} and \ref{thm:infinite-variation} is as follows:
\begin{itemize}
	\item if $\alpha \geq 1$ and $\eta \geq 2$, then
	$1/\gamma = {1}/{2}$; 
	\item if $\alpha \geq 1$, $\eta < 2$ or
	$\alpha < 1$, $\eta \leq 1 + \alpha$, $\beta \leq 1 + \alpha$, then 
	$1/\gamma = {1}/{\eta}$;
	\item if $\alpha < 1$, $\eta > 1 + \alpha$, $\beta \leq 1 + \alpha$, then 
	$1/\gamma = {1}/(1+\alpha)$; 
	\item if $\alpha < 1$, $\eta > 1 + \alpha$, $\beta > 1 + \alpha$, then 
	$1/\gamma > 1-\alpha/\beta$;
	\item if $\alpha < 1$, $\eta \leq 1 + \alpha < \beta$, $\eta < ( 1 - \alpha/\beta)^{-1}$,
	then $1/\gamma = 1/\eta$;
	\item if $\alpha < 1$, $\eta \leq 1 + \alpha < \beta$, $\eta \geq ( 1 - {\alpha}/{\beta})^{-1}$,
	then $1/\gamma > 1-\alpha/\beta$.
\end{itemize}
For the last three cases, recall that in Theorem~\ref{thm:infinite-variation} (ii) 
the corresponding integral is taken on $|z| > 1$, therefore the integrals on $|z| \in (0,1]$ in
\eqref{eq:main-rates} are missing. Furthermore, in Theorem~\ref{thm:infinite-variation} (ii)
{$\gamma < (1 - \alpha/\beta)^{-1}$}. A similar statement holds true if $\pi(V) = \infty$ 
with a few more cases, and an additional parameter describing the tail of $\pi$.

This gives bounds on the rate of growth of the integrated process in terms of the behavior of 
$\pi$ and $\lambda$ at zero and infinity. From the limit theorems for supOU processes
\cite{grahovac2019, grahovac2020} and trawl processes \cite{pakkanen2021,talarczyk2020}, 
one can conclude that the given bounds are sharp for these classes of processes. 
We note that in \cite{grahovac2020}, there is an error in the proof of Theorem 1 in the case
corresponding to $\alpha < 1$ and $\eta \leq 1 + \alpha < \beta$ here. The limiting process in \cite{grahovac2020}
{can be an $\eta$-stable L\'evy process} or a $\beta$-stable $1-\alpha/\beta$-self-similar process, 
depending on whether $\eta < ( 1 - \alpha/\beta)^{-1}$ or $\eta > ( 1 - {\alpha}/{\beta})^{-1}$.

\section{Examples} \label{sec:examples}

The stated results hold for several well-known classes of mixed moving average processes, 
including the supOU and the trawl processes, as well as the less-studied supfOU processes.
We show that these processes satisfy the Assumptions \ref{assum-1} and \ref{assum-2}, 
and state the corresponding results. Additionally, we consider the case of L\'evy-driven moving averages.

\subsection{SupOU process}

Suppose that $V = (0, \infty)$ and $f(x, v) = e^{-xv} \bone(v \geq 0)$. 
Then the process defined in \eqref{eq:MMA-definition} is the supOU process
\cite{barndorff1998invGaussian, barndorff2000supOU}. It is well defined if and only if 
$m_{-1}(\pi) < \infty$ and $\int_{|z| > 1} \log |z| \lambda (\dd z) < \infty$.
This example was previously investigated in \cite{grahovac2025}, where the corresponding statements of 
Theorems~\ref{thm:finite-variation} and \ref{thm:infinite-variation} were given. For $\gamma = 1$,
we slightly improve the results, as no additional moment assumption is needed. 
Here, we simply verify that our assumptions hold.

When $\E|X(1)|=\infty$, conditions \eqref{eq:fubini-cond-1} and
	\eqref{eq:fubini-cond-3} hold by Lemma \ref{lemma:sufficient-fubini}, while
	\eqref{eq:fubini-cond-2} holds if $\pi$ is a finite measure. The weaker
	condition
	\begin{equation}\label{eq:fubini-supOU-suff}
		\int_{(1,\infty)} \int_{|z| > 1}
		( 1 \wedge |z|x^{-1} )
		\lambda(\dd z) \pi (\dd x) < \infty
	\end{equation}
	already suffices, and this is what we assume. We note that it is missing in
	\cite{grahovac2025}.

{The integrated process has decomposition \eqref{eq:decomposition-X*}} 
with $f_2(x, u) = x^{-1}e^{-xu}$ and $f_1(x) = x^{-1}$. 
{Assumption~\ref{assum-1} holds with $K = 1$, since
$\int_0^\infty s e^{-xs}\dd s = x^{-2} = f_1(x)^2$. Moreover, since
$f \leq 1$, Assumption~\ref{assum-2} follows from
Lemma~\ref{lemma:equiv-A1-A2}.}

\subsection{Trawl process}

Suppose that $\psi \colon [0, \infty) \to (0, \psi(0)]$ is an integrable, 
non-increasing, and continuous function with $\psi(0) < \infty$. 
Let $V = (0, \psi(0)]$, $\pi  \equiv \Leb$ and 
$f(x, v) = \bone(v \geq 0, \, 0 \leq x \leq \psi(v))$. 
Then \eqref{eq:MMA-definition} defines the  monotonic trawl process, 
see \cite[Chapter 8]{barndorff2018book}. The set 
$A_t \coloneqq \{ (x, s) \colon s \leq t, \, 0 \leq x \leq \psi(t-s) \}$	
is commonly referred to as the trawl set, while the function $\psi$ is called the trawl function. 

Trawl processes were first introduced in the work of Barndorff-Nielsen \cite{barndorff2011statID} 
as a class of stationary infinitely divisible stochastic processes. As explained in \cite{barndorff2014intTrawl},
the trawl set $A_0$ can be regarded as a fishing net dragged along the sea, 
so that at time $t$ it is in position $A_t$. A similar structure was considered by Wolpert and Taqqu
\cite{wolpert2005fou}, who referred to such processes as upstairs representations.	

When $\E|X(1)|=\infty$, the conditions \eqref{eq:fubini-cond-1}--\eqref{eq:fubini-cond-3}
of the stochastic Fubini theorem hold by Lemma \ref{lemma:sufficient-fubini}. 
Note that {$\psi$ admits a generalized inverse}, as it is continuous and non-increasing. 
Let $\psi^\leftarrow(x) = \inf\{u \geq 0 \colon \psi(u) \leq x\}$ denote its right-continuous 
(generalized) inverse function. Then $\psi(t) \geq x$ if and only if $t \leq \psi^\leftarrow(x-)$, 
where $\psi^\leftarrow(x-)$ stands for the left limit at $x$. Therefore, the integrated process has
decomposition \eqref{eq:decomposition-X*} with 
\[
f_2(x, u) = (\psi^{\leftarrow}(x-) - u)\bone (\psi^\leftarrow(x-) > u)
\]
and $f_1(x) = \psi^{\leftarrow}(x-)$.
{Assumption~\ref{assum-1} holds with $K = 1/2$, since
\[
\int_0^{\psi^\leftarrow(x-)} s \, \dd s = \frac{f_1(x)^2}{2}.
\]
Moreover, since
$f \leq 1$, Assumption~\ref{assum-2} follows from
Lemma~\ref{lemma:equiv-A1-A2}.} 
{Therefore, Theorems~\ref{thm:finite-variation} and \ref{thm:infinite-variation} apply, 
and condition \eqref{eq:Cgamma-condition} specializes to}
\begin{equation*}
	\int_0^{\psi(0)} \int_{\R} 
	|z|^\gamma \psi^\leftarrow(x-)^\gamma 
	\bone(|z|\psi^\leftarrow(x-) > 1) 
	\lambda (\dd z) \dd x  < \infty.
\end{equation*}

\subsection{SupfOU processes}\label{subsec:supfOU}

Suppose that $V = (0, \infty)$. Motivated by the fractional Ornstein--Uhlenbeck processes
introduced by Wolpert and Taqqu \cite[Section 4.2.2]{wolpert2005fou}, for $\kappa > 0$, 
we consider a kernel of the form 
\[
f(x, v) = \frac{1}{\Gamma(\kappa)} (xv)^{\kappa - 1} e^{-xv} \bone (v > 0).
\]
We will call the process defined by \eqref{eq:MMA-definition} with such a kernel a superposition of
fractional Ornstein--Uhlenbeck processes (supfOU process). An application of 
\cite[Proposition 34]{barndorff2018book} shows that such a process is well-defined.

\begin{proposition} \label{prop:supfou-existence}
	The supfOU process exists if and only if $m_{-1}(\pi) < \infty$, 
	\begin{equation*}
		\int_{|z| > 1} \log |z| \lambda(\dd z) < \infty,	
	\end{equation*}
	and one of the following holds
	\begin{itemize}
		\item[(i)] $\kappa > 1/2$, or 
		\item[(ii)] $\kappa = 1/2$ and $\int_{|z| \leq 1} z^2 \log |z|^{-1} \lambda (\dd z ) < \infty$, or
		\item[(iii)] $\kappa < 1/2$ and $\int_{|z| \leq 1} |z|^{\frac{1}{1-\kappa}} \lambda(\dd z) < \infty$.
	\end{itemize}
\end{proposition}

We check the conditions \eqref{eq:fubini-cond-1}--\eqref{eq:fubini-cond-3}, 
which are needed when $\E|X(1)|=\infty$. For $\kappa \geq 1$ the kernel $f$ is bounded, hence \eqref{eq:fubini-cond-1} holds by 
Lemma \ref{lemma:sufficient-fubini}. 
For $0<\kappa<1$, a calculation similar to that in 
Lemma \ref{lemma:supfOU-2} shows that
\eqref{eq:fubini-cond-1} holds provided that
\begin{equation*}
	\int_{|z|\leq 1}|z|^{1/(1-\kappa)}\lambda(\dd z)<\infty,	
\end{equation*}
which is implied by $\int_{|z|\leq1}z^2\lambda(\dd z)<\infty$ when
$\kappa\geq1/2$. For $\kappa < 1/2$ we assume this to hold when $\E|X(1)|=\infty$.

For $\kappa \leq 1$, $f$ is non-increasing, hence \eqref{eq:fubini-cond-3} holds by 
Lemma \ref{lemma:sufficient-fubini}. 
When $\kappa>1$, \eqref{eq:fubini-cond-3} holds if we take
\begin{equation*}
	g(x,s) = 
	\begin{cases}
		\frac{1}{\Gamma(\kappa)} (\kappa-1)^{\kappa - 1} e^{-\kappa+1}, & 
		s \leq (\kappa-1)x^{-1},\\
		f(x,s), & 
		s > (\kappa-1)x^{-1},\\
	\end{cases}
\end{equation*}
since $f(x,\cdot)$ is increasing and bounded on $(0,(\kappa-1)x^{-1})$ and non-increasing thereafter.
Moreover, 
\begin{align*}
& \int_V \int_0^{\infty} \int_{|z|>1}
\left( 1 \wedge |z| g(x,s) \right) \lambda(\dd z) \dd s \pi(\dd x) \\
& \leq 
\int_V \int_0^{\infty} \int_{|z|>1}
\left( 1 \wedge |z| f(x,s) \right) \lambda(\dd z) \dd s \pi(\dd x)  \\
& \quad + \int_V \frac{\kappa -1}{x} \pi(\dd x) {\lambda(\{|z| > 1 \})} < \infty.
\end{align*}
Finally, \eqref{eq:fubini-cond-2} holds under \eqref{eq:fubini-supOU-suff},
which we assume here as for the supOU process.

The integrated process has decomposition \eqref{eq:decomposition-X*} with 
$f_2(x, u) =  x^{-1} \frac{\Gamma_{\kappa}(xu)}{\Gamma(\kappa)}$ and $f_1(x) = x^{-1}$, 
where $\Gamma_{\kappa}(x)= \int_x^\infty y^{\kappa - 1} e^{-y} \dd y$ represents the upper incomplete gamma function. 
{
Assumption~\ref{assum-1} holds with $K=\kappa$, since
\[
\int_0^\infty s f(x,s)\dd s = \kappa x^{-2} = \kappa f_1(x)^2.
\]
If $\kappa \geq 1$, then $f$ is bounded, and hence
Assumption~\ref{assum-2} follows from Lemma~\ref{lemma:equiv-A1-A2}. 
For $0<\kappa<1$, $f$ is unbounded, so we verify Assumption~\ref{assum-2} directly. For
condition~\eqref{eq:f-cond-int1}, we have
\begin{equation*}
	\int_0^\infty \left(f_2(x,u)-f_2(x,t+u)\right)^p\dd u
	= \frac{x^{-p-1}}{\Gamma(\kappa)^p}
	\int_0^\infty \left( \int_s^{tx+s} y^{\kappa-1}e^{-y}\dd y \right)^p\dd s.
\end{equation*}
Thus, setting $r=tx$, it suffices to bound
\begin{equation*}
	\int_0^\infty \left( \int_s^{r+s} y^{\kappa-1}e^{-y}\dd y \right)^p \dd s
\end{equation*}
by a constant multiple of $(r \wedge 1)^p$. For $r \geq 1$ this is immediate,
while for $r\downarrow0$,
\begin{equation*}
	\lim_{r\downarrow0}
	r^{-p} \int_0^\infty \left( \int_s^{r+s} y^{\kappa-1}e^{-y}\dd y\right)^p\dd s
	= \int_0^\infty \left(s^{\kappa-1}e^{-s}\right)^p\dd s.
\end{equation*}
The latter integral is finite whenever $\kappa>1-p^{-1}$. Hence
\eqref{eq:f-cond-int1} holds for every $p\in[1,2]$ if $\kappa>1/2$.

On the other hand, condition~\eqref{eq:f-cond-int2} fails for every
$0<\kappa<1$. Indeed, it is equivalent to the boundedness of
\begin{equation*}
	r^{-p-1} \int_0^r \left( \int_0^s y^{\kappa-1}e^{-y}\dd y \right)^p \dd s.
\end{equation*}
Since
\begin{equation*}
	\lim_{s\downarrow0} s^{-\kappa} \int_0^s y^{\kappa-1}e^{-y}\dd y
	= \frac{1}{\kappa},
\end{equation*}
the above expression diverges as $r\downarrow0$ whenever $\kappa<1$.
Therefore, Assumption~\ref{assum-2} does not hold for $0<\kappa<1$.}

{Consequently, Theorem~\ref{thm:finite-variation} holds for all $\kappa > 0$ and
Theorem~\ref{thm:infinite-variation} for $\kappa \geq 1$,
with condition \eqref{eq:Cgamma-condition} specialized to} 
\begin{equation*} 
	\int_{(0, \infty)} \int_{\R} 
	|z|^\gamma x^{-\gamma} \bone(|z|x^{-1} > 1) 
	\lambda (\dd z) \pi(\dd x)  < \infty.
\end{equation*}

\subsection{Rescaled kernels}

The supOU, supfOU and trawl kernels share a common structure, with the mixing
parameter entering through the time scale. This suggests the following family.
\begin{equation} \label{eq:scale-kernel}
	f(x,s)= A(x) h\left(\frac{s}{q(x)}\right),
	\qquad x\in V,\ s\geq0,
\end{equation}
where $q\colon V\to(0,\infty)$ and $A\colon V\to(0,\infty)$ are measurable
functions and $h\colon[0,\infty)\to[0,\infty)$ is measurable, $q(x)$
determines the time scale of the kernel and $A(x)$ its amplitude, while the
shape is given by $h$. 
Suppose that $0<\int_0^\infty h(u)\dd u<\infty$,
$\int_0^\infty u h(u)\dd u<\infty$ and $\inf_{x\in V}A(x)>0$.
Substituting $u=s/q(x)$ we obtain
\[
f_1(x)=A(x)q(x)\int_0^\infty h(u)\dd u
\]
and
\[
\int_0^\infty s f(x,s)\dd s = A(x)q(x)^2\int_0^\infty u h(u)\dd u.
\]
The time scale cancels in the ratio and
\begin{equation*}
	\frac{\int_0^\infty s f(x,s)\dd s}{f_1(x)^2}
	=\frac{1}{A(x)}
	\frac{\int_0^\infty u h(u)\dd u}{\left(\int_0^\infty h(u)\dd u\right)^2}
	\leq \frac{1}{\inf_{x\in V}A(x)}
	\frac{\int_0^\infty u h(u)\dd u}{\left(\int_0^\infty h(u)\dd u\right)^2},
\end{equation*}
showing that Assumption~\ref{assum-1} holds. If in addition $h$ 
and $A$ are bounded, then $f$ is bounded and
Lemma~\ref{lemma:equiv-A1-A2} gives Assumption~\ref{assum-2} as well.

\subsection{Moving average processes}

If $V$ is a one-point space, $\Lambda$ becomes a random measure generated by the increments
of a two-sided L\'evy process and the $x$ component in \eqref{eq:MMA-definition} may be ignored. 
Moreover, the function $f_1$ reduces to a finite constant, denoted by $I_f$, while
$f_2(u) = \int_u^\infty f(v) \dd v$, for $u \geq 0$. 
We assume that conditions \eqref{eq:fubini-cond-1}, \eqref{eq:fubini-cond-2} and \eqref{eq:fubini-cond-3} hold.
Note that \eqref{eq:fubini-cond-2} is trivial in this case.

This special case of $X$ now being a moving average process is simpler to handle and 
the general statements of Theorems~\ref{thm:finite-variation} and \ref{thm:infinite-variation} 
can be written in a more direct form. 
{In particular, condition \eqref{eq:Cgamma-condition} reduces to a moment assumption 
\begin{equation} \label{eq:main-MA}
	\int_{|z| > 1} |z|^\gamma \lambda(\dd z) < \infty,
\end{equation}
while Assumptions~\ref{assum-1} and~\ref{assum-2}
specialize, respectively, to the following.

\begin{assumptionprime} \label{assum-1-ma}
	There exists $K > 0$ such that $\int_0^\infty s f(s)\dd s \leq K I_f^2$.
\end{assumptionprime}

\begin{assumptionprime} \label{assum-2-ma}
	For each $p \in [1,2]$ there exists $c_p > 0$ such that for all $t > 0$
	\begin{equation*}
		\int_0^\infty (f_2(u) - f_2(t+u))^p \dd u
		\leq c_p (t \wedge I_f)^p I_f,\quad
		\int_0^t (I_f - f_2(u))^p \dd u
		\leq c_p (t \wedge I_f)^p t.
	\end{equation*}
\end{assumptionprime}

\begin{cor}
	Assume that $\int_{|z| \leq 1} |z| \lambda(\dd z) < \infty$. 
	Let $\gamma \in (0,2]$, {$a = \int_{|z|\leq 1} z \lambda(\dd z)$} 
	and $b = 0$. Moreover, assume that Assumption~\ref{assum-1-ma} holds if $\gamma > 1$. Then
	\eqref{eq:main-MA} with $\gamma < 2$ implies \eqref{eq:finite-variation-result-1}, while
	\eqref{eq:main-MA} with $\gamma = 2$ implies \eqref{eq:finite-variation-result-2}.
\end{cor}
Since $I_f$ is a finite constant, $\alpha_0 = \infty$. 
Thus, the second result then simplifies, as $\beta > 1 + \alpha$ can no longer occur. 
\begin{cor}
	Assume that $\int_{|z|\leq 1}|z|\lambda(\dd z) = \infty$ and 
	let $\gamma \in (0,2)$, and $a = b = 0$. Moreover, assume that 
	Assumption~\ref{assum-2-ma} holds. Then \eqref{eq:main-MA} 
	implies \eqref{eq:infinite-variation-result-1}.
\end{cor}
}

These statements hold for the non-mixed versions of the preceding examples,
namely for the Ornstein--Uhlenbeck process, where
\[
f(v) = e^{-x v} \bone(v \geq 0), \quad  x>0,
\]
and for the fractional
Ornstein--Uhlenbeck process, where
\[
f(v) = \frac{1}{\Gamma(\kappa)} (xv)^{\kappa-1} e^{-xv}, \quad x> 0,
\]
$\kappa \geq 1$.
The latter kernel is often referred to as the gamma kernel; 
see \cite{ barndorff2011bss, barndorff2013bss-limit, barndorff2009bss-volatility, basse2009semimart}.

{In these examples the kernel decays exponentially and the assumptions
are easily verified. A power kernel shows where they cease to hold. Take
\[
f(s) = (1+s)^{-\kappa}
\]
with $\kappa > 1$, so that $f \in L^1(\Leb)$,
$I_f = (\kappa-1)^{-1}$ and 
\[
f_2(u) = (1+u)^{-(\kappa-1)}/(\kappa-1).
\]
Substituting $w = 1+s$, we get 
\[
\int_0^\infty s f(s) \dd s = 
\int_1^\infty \left(w^{1-\kappa} - w^{-\kappa}\right) \dd w,
\]
which is finite if and only if $\kappa > 2$, in which case it equals
$\left((\kappa-1)(\kappa-2)\right)^{-1}$. Thus for $\kappa > 2$
Assumption~\ref{assum-1-ma} holds with $K = (\kappa-1)(\kappa-2)^{-1}$, and since
$f \leq 1$, Lemma~\ref{lemma:equiv-A1-A2} gives Assumption~\ref{assum-2-ma}
as well. For $\kappa \in (1,2]$ the process is still well defined, but the
first moment of the kernel is infinite and both assumptions fail. The
assumptions thus exclude kernels whose temporal decay is too slow.
}

\section{Proofs} \label{sec:proofs}

We use the following conventions and notations throughout the proofs. 
Non-specified limit operations are meant as $n \to \infty$ or $t \to \infty$. 
Constants $c > 0$ are finite positive constants whose value does not depend 
on relevant quantities, and can be different at each appearance.

\subsection{A stochastic Fubini theorem}

\begin{proof}[Proof of Lemma \ref{lemma:StochFubini}]
	Assume first that $\E |X(1)| < \infty$. We can and do further assume that $\Lambda$ is centered. Indeed, the general case follows from the centered case combined with the deterministic version of Fubini's theorem, as $f \in L^1(\pi \times \Leb)$.
	
	By Theorem 3.1 and Remark 3.2 in \cite{barndorff2011qou}, it suffices to show that for any $u \in [0, t]$, the map $(x, s) \mapsto f(x, u -s)$ belongs to the Musielak--Orlicz space $L^{\phi}$ (see \cite{rajput1989} for details), i.e.
\begin{equation} \label{eq:fubini1}
\begin{split}
I_1 & =  \int_V \int_{\R}  \Big[
b f(x, u-s)^2 \\
& \qquad + \int_{\R} \left(
z^2 f(x, u-s)^2 \wedge 
|z| f(x, u-s) 
\right) \lambda(\dd z)
\Big] \dd s \pi(\dd x)  < \infty,
\end{split}
\end{equation}
	and that
\begin{equation} \label{eq:fubini2}
\begin{split}
I_2 = & \int_0^t \int_V \int_{\R}  \Big[
		b f(x, u-s)^2  \\
& 	+ \int_{\R}  \left(
z^2 f(x, u-s)^2 \wedge 
|z| f(x, u-s) 
\right) \lambda(\dd z) 
\Big] \dd s \pi(\dd x) 
\dd u < \infty.
\end{split}
\end{equation}
	By changing variables, we get $I_2 = t I_1$. Thus, \eqref{eq:fubini1} directly implies \eqref{eq:fubini2}. Since we assume $f$ is $\Lambda$-integrable, the conditions in  \cite[Theorem 2.7]{rajput1989} hold. Moreover, following the expression for the L\'evy measure of integrals with respect to $\Lambda$ from \cite[Proposition 35]{barndorff2018book}, we have that $\E |X(1)|<\infty$ if and only if
	\begin{equation*}\label{eq:X-mom-cond}
		\int_V \int_\R \int_\R |z| f(x,s) \bone(|z| f(x,s) > 1) \lambda(\dd z) \dd s \pi(\dd x)< \infty.
	\end{equation*}
	Therefore, $I_1 < \infty$.
	
Suppose now that $\E |X(1)| = \infty$. We use the L\'evy--It\^o decomposition 
in \eqref{eq:levy-ito}. We further decompose
	\begin{equation*}
		\Lambda(A) = \Lambda_1(A) + \Lambda_2(A), \quad A \in \mathcal{B}_0(V \times \R),
	\end{equation*}
	where $\Lambda_1$ and $\Lambda_2$ are independent random measures such that 
	$\Lambda_1$ has a characteristic quadruple $(a, b, \lambda_1, \pi)$, with 
	$\lambda_1(\dd z) =  \bone(|z| \leq 1) \lambda(\dd z)$ and $\Lambda_2$ has characteristic quadruple $(0, 0, \lambda_2, \pi)$, with 
	$\lambda_2(\dd z) = \bone(|z| > 1) \lambda(\dd z)$. 
	As in the previous case, we can assume that $\Lambda_1$ is centered. Using \cite[Proposition 3.3.9]{samorodnitsky2016}, we have that almost surely
	\begin{equation} \label{eq:Fubini-decomp}
    \begin{split}
		X(u) & = 
		\int_{V \times \R} f(x, u-s) \Lambda_1(\dd x, \dd s) +
		\int_{V \times \R}  f(x, u-s) \Lambda_2(\dd x, \dd s) \\
        & 	\eqqcolon X_1(u) + X_2(u) .
    \end{split}
	\end{equation}
	By the same argument as for $X$, $\E |X_1(1)|<\infty$ if and only if \eqref{eq:fubini-cond-1} holds. Hence, the previous case applies to $X_1$.
	
	For $X_2$, we have by the deterministic Fubini's theorem that a.s.
	\begin{equation*}
		\int_0^t X_2(u) \dd u = \int_{V \times \R \times (\R \backslash [-1,1])} 
		\int_0^t z f(x,u-s) \dd u \mu(\dd x, \dd s, \dd z),
	\end{equation*}
	whenever the latter integral exists a.s. The necessary and sufficient 
	condition for this is 
	\begin{equation}\label{eq:fubini-cond-extra2}
		\int_V \int_\R \int_{|z|>1}
		\left( 1  \wedge |z| \int_0^t f(x,u-s) \dd u \right) \lambda(\dd z) \dd s \pi(\dd x) < \infty.
	\end{equation}
	It remains to show that the assumptions imply \eqref{eq:fubini-cond-extra2}. For $s<0$ we have by \eqref{eq:fubini-cond-3} that
	\begin{equation*}
		\begin{split}
			&\int_V \int_{-\infty}^0 \int_{|z|>1} 
			\left( 1  \wedge |z| \int_0^t f(x,u-s) \dd u \right) \lambda(\dd z) \dd s \pi(\dd x)\\
			&\leq \int_V \int_{-\infty}^0 \int_{|z|>1}
			\left( 1  \wedge |z| t g(x,-s) \right) \lambda(\dd z) \dd s \pi(\dd x) < \infty,
		\end{split}
	\end{equation*}
	and when $s\in (0,t)$, from \eqref{eq:fubini-cond-2} we obtain
	\begin{equation*}
		\begin{split}
			&\int_V \int_0^t \int_{|z|>1}
			\left( 1  \wedge |z| \int_0^t f(x,u-s) \dd u \right) \lambda(\dd z) \dd s \pi(\dd x)\\
			&\leq t \int_V \int_{|z| > 1}
			\left( 1  \wedge |z| f_1(x)  \right) \bone(f_1(x) \leq 1)
			\lambda(\dd z) \pi(\dd x)  \\ 
			& \quad + 
			t \int_V \int_{|z| > 1}
			\left( 1  \wedge |z| f_1(x) \right) \bone(f_1(x) > 1) 
			\lambda(\dd z) \pi(\dd x).
		\end{split}
	\end{equation*}
	The first integral is finite by \eqref{eq:fubini-cond-2}, while for the second we have
	\begin{align*}
		&\int_V \int_{|z| > 1}
		\left( 1  \wedge |z| f_1(x) \right) \bone(f_1(x) > 1)
		\lambda(\dd z) \pi(\dd x) \\ 
		&\leq \int_V \int_{|z| > 1}
		f_1(x) f_1(x)^{-1} \bone(f_1(x) > 1)
		\lambda(\dd z) \pi(\dd x) \\ 
		& \leq \lambda(\{z \colon |z| > 1\})
		\int_V f_1(x) \pi(\dd x) < \infty.
	\end{align*}
	Hence, \eqref{eq:fubini-cond-extra2} holds.
\end{proof} 

The conditions \eqref{eq:fubini-cond-1}--\eqref{eq:fubini-cond-3} may seem complicated at first, but they are implied by many simpler conditions. We provide below several sufficient conditions that can be readily checked.
\begin{lemma} \label{lemma:sufficient-fubini}
	\begin{enumerate}[label=(\roman*)]
		\item If $f$ is bounded, then \eqref{eq:fubini-cond-1} holds.
		\item If $\pi$ is a finite measure, then \eqref{eq:fubini-cond-2} holds.
		\item  {If $f$ is non-increasing in the second variable}, then \eqref{eq:fubini-cond-3} holds with $g=f$.
	\end{enumerate}
\end{lemma}

\begin{proof}
	For (i), let $\sup_{x,s} f(x,s) = K < \infty$. Then, since $f \in L^1(\pi \times \Leb)$,
	\begin{align*}
		& \int_V \int_\R \int_{|z|\leq 1} |z| f(x,s) \bone(|z| f(x,s) > 1)
		\lambda(\dd z) \dd s \pi(\dd x)   \\
		& \leq 
		\int_V \int_\R  f(x,s)  \dd s \pi(\dd x)
		\int_{|z|\leq 1} |z | \bone(|z| \geq  1 / K)  \lambda(\dd z) < \infty.
	\end{align*}
	From \eqref{eq:fubini-cond-2}, (ii) is straightforward. If $f$ is non-increasing, then we have $\int_0^t f(x,u-s) \dd u \leq t f(x,-s)$, and the second condition in \eqref{eq:fubini-cond-3} follows from integrability conditions \eqref{eq:RR-cond}, which proves (iii).
\end{proof}

\subsection{Relation between the assumptions}

\begin{proof}[Proof of Lemma \ref{lemma:equiv-A1-A2}]
Assume first that Assumption~\ref{assum-1} holds and that $f$ is bounded by $M > 1$. Since
\begin{equation*}
	f_2(x,u)-f_2(x,u+t)
	= \int_u^{u+t}f(x,s)\dd s
	\leq M\left(t\wedge f_1(x)\right),
\end{equation*}
	by Fubini's theorem and Assumption~\ref{assum-1},
	\begin{align*}
		\int_0^\infty
		\left(f_2(x,u)-f_2(x,u+t)\right)\dd u
		&= \int_0^\infty (t\wedge s)f(x,s)\dd s\\
		&\leq (tf_1(x))\wedge(Kf_1(x)^2)\\
		&\leq (1+K)(t\wedge f_1(x))f_1(x).
	\end{align*}
	Therefore, for every $p\in[1,2]$,
	\begin{align*}
		&\int_0^\infty
		\left(f_2(x,u)-f_2(x,u+t)\right)^p\dd u \\
		&\leq M^{p-1}(t\wedge f_1(x))^{p-1}
		\int_0^\infty
		\left(f_2(x,u)-f_2(x,u+t)\right)\dd u\\
		&\leq M^{p-1}(1+K)
		(t\wedge f_1(x))^p f_1(x),
	\end{align*}
	which proves~\eqref{eq:f-cond-int1}. 
	For the second condition, observe that for $0 \leq u \leq t$
	\begin{equation*}
		f_1(x)-f_2(x,u)
		= \int_0^u f(x,s)\dd s
		\leq M (u\wedge f_1(x)) 
		\leq M (t \wedge f_1(x)).
	\end{equation*}
	It follows that, for every $p\in[1,2]$,
	\begin{equation*}
		\int_0^t \left(f_1(x)-f_2(x,u)\right)^p\dd u
		\leq M^p \left(t\wedge f_1(x)\right)^p t,
	\end{equation*}
	which shows~\eqref{eq:f-cond-int2}.
	
	Conversely, suppose that \eqref{eq:f-cond-int1} holds with $p=1$. By using Fubini's theorem, we obtain
	\begin{equation*}
    \begin{split}
		\int_0^\infty (t\wedge s)f(x,s)\dd s
		& = \int_0^\infty \left(f_2(x,u)-f_2(x,u+t)\right)\dd u \\
		& \leq c_1\left(t\wedge f_1(x)\right)f_1(x)
		\leq c_1f_1(x)^2.
	\end{split}
    \end{equation*}
	Letting $t\to\infty$ and applying the monotone convergence theorem shows that Assumption~\ref{assum-1} holds.
\end{proof}

\subsection{Proof of Theorem~\ref{thm:finite-variation}}

The claim in the case $\gamma = 1$ follows from the 
mixing property of the MMA process \cite[Theorem 3.5]{fuchs2013mixing}, 
which implies $t^{-1} X^*(t) \to \E X(1)$ a.s.~(see e.g.~\cite[Corollary 25.9]{kallenberg2021foundations}).
Therefore, in the remainder of this section we consider $\gamma \neq 1$. 
Recall that under our assumption $\Lambda$ has the form \eqref{eq:Lambda-finite-variation}.
Assume first that  $\lambda((-\infty,0)) = 0$. Thus, \eqref{eq:decomposition-X*} reads as
\begin{equation} \label{eq:decomposition-finite-variation}
	\begin{split}
		X^*(t) &=  
		\int_{V \times (-\infty, 0] \times (0, \infty)}
		z (f_2(x, -s) - f_2(x, t-s)) \mu(\dd x, \dd s, \dd z)  \\ 
		& \quad + \int_{V \times (0, t] \times (0, \infty)}
		z f_1(x) \mu(\dd x, \dd s, \dd z) \\ 
		& \quad - \int_{V \times (0, t] \times (0, \infty)}
		z f_2(x, t-s) \mu(\dd x, \dd s, \dd z)	 \\
        & 		\eqqcolon X_-^*(t) + X_{+, 1}^*(t) - X_{+,2}^*(t).
	\end{split}	
\end{equation}
The next lemma shows that $X_{-}^*(t)$ is negligible in the limit 
under the assumptions of Theorem~\ref{thm:finite-variation}.

\begin{lemma} \label{lemma:finite-variation-1}
	Let $\gamma \in (0, 2]$. If $\gamma \in (1, 2]$, suppose that 
	$f_1 \in L^\gamma(\pi )$, $m_1(\lambda ) < \infty$ 
	and that Assumption~\ref{assum-1} holds. Then 
	\begin{equation*}
		\lim_{t \to \infty} \frac{X^*_-(t)}{t^{1/\gamma}} = 0 \quad \text{a.s.}
	\end{equation*}		
\end{lemma}

\begin{proof}
	First, assume $\gamma<1$. If $\E X(1) < \infty$, the claim follows from 
	the fact that $t^{-1} X^*_-(t) \leq t^{-1} X^*(t) \to \E X(1)$ a.s., as $t \to \infty$.
	Assume $\E X(1) = \infty$ and decompose $X_{-}(t) \eqqcolon X_{-,1}(t) + X_{-,2}(t)$, where 
	\begin{align*}
		X_{-, 1}(t) &= \int_{V \times (-\infty, 0] \times (0, 1]}
		zf(x, t-s) \mu(\dd x, \dd s, \dd z), \\ 
		X_{-, 2}(t) &= \int_{V \times (-\infty, 0] \times (1, \infty)}
		zf(x, t-s) \mu(\dd x, \dd s, \dd z).
	\end{align*}	
	{Then
	\begin{equation*}
		X_{-,1}(t) \leq \int_{V \times (-\infty, t] \times (0, 1]}
		zf(x, t-s) \mu(\dd x, \dd s, \dd z) \eqqcolon Y(t),	
	\end{equation*}
	where $Y$ is an MMA process, and 
    $\E Y(1) = \int_V f_1(x) \pi(\dd x) \int_{(0,1]} z \lambda(\dd z) <\infty$. 
	From the mixing property \cite[Theorem 3.5]{fuchs2013mixing} we get that $t^{-1} \int_0^t Y(s) \dd s \to \E Y(1)$ a.s., implying that 
	$t^{-1} X^{*}_{-,1} (t) \to 0$ a.s.} For $X_{-,2}$ 
	\begin{align*}
		t^{-1} X^*_{-,2}(t) &= \int_{V \times (-\infty, 0] \times (1,\infty)}
		z t^{-1} \int_0^t f(x,u-s) \dd u \mu(\dd x, \dd s, \dd z)\\
		& \leq \int_{V \times (-\infty, 0] \times (1,\infty)}
		z g(x,-s) \mu(\dd x, \dd s, \dd z),
	\end{align*}
	which is finite by \eqref{eq:fubini-cond-3}. 
	{Thus, $t^{-1} X^*_-(t)$ is almost surely bounded, and since $\gamma<1$, the result follows.}
	
	{Now assume that $\gamma \in (1, 2]$. We have by Fubini's theorem
	\begin{align*} 
		\E X^*_-(t) 
		&= \int_{V \times (-\infty, 0] \times (0,\infty)}
		z(f_2(x, -s) - f_2(x, t-s)) \nu(\dd x, \dd s, \dd z) \\ 
		&= m_1(\lambda ) \int_V \left( 
		\int_0^\infty \int_u^{t+u} f(x, v) \dd v \dd u 
		\right) \pi (\dd x)  \\ 
		&= m_1(\lambda) \int_V \int_0^\infty (t \wedge v) f(x,v) \dd v \pi(\dd x).  
	\end{align*}
	By Assumption~\ref{assum-1}, the inner integral $I$ is bounded as
	\begin{align*}
		I & \leq (t f_1(x)) \wedge \int_0^\infty v f(x, v) \dd v 
		\leq (t f_1(x)) \wedge (K f_1(x)^2) \\ 
		& \leq (t f_1(x))^{2-\gamma} (K f_1(x)^2)^{\gamma - 1} 
		\leq K^{\gamma-1} t^{2-\gamma} f_1(x)^\gamma,
	\end{align*} 	
	and thus $\E X_-^*(t) \leq c t^{2-\gamma}$.
	Markov's inequality implies that for $\varepsilon > 0$
	\begin{equation*}
		\P(X^*_-(t)  > \varepsilon t^{1/\gamma}) \leq 
		\frac{\E X^*_-(t) }{\varepsilon t^{1/\gamma}} \leq
		c t^{2 - \gamma - 1/\gamma}.
	\end{equation*}
	For $a > 0$ such that $2 - \gamma - 1/\gamma < -1/a$, set $t_n \coloneqq n^a$. 
	The Borel--Cantelli lemma then implies $\lim_{n \to \infty} t_n^{-1/\gamma} X^*_-(t_n) = 0$ a.s.
	The result follows by monotonicity, as for $t \in [t_n, t_{n+1}]$ 
	\begin{equation*}
		t^{-1/\gamma} X^*_-(t) 
		\leq t_n^{-1/\gamma} X^*_-(t_{n+1})  
		= t_{n+1}^{-1/\gamma} X^*_-(t_{n+1}) 
		\left(\frac{t_{n+1}}{t_n}\right)^{1/\gamma} 
		\to 0 \quad \text{a.s.}
	\end{equation*}	
	}
\end{proof}
 
We now turn to $X^*_{+, 1}(t)$, which is a subordinator. Its behavior constitutes the main component
of the asymptotics in Theorem~\ref{thm:finite-variation} and it is 
directly related to condition \eqref{eq:Cgamma-condition}. 

\begin{lemma} \label{lemma:finite-variation-2}
	Assume that for some $\gamma \in (0,2]$ 
	\begin{equation} \label{eq:Cgamma-condition-lemma}
		\int_V \int_{(0,\infty)} 
		z^\gamma f_1(x)^\gamma \bone (zf_1(x) > 1) 
		\lambda(\dd z) \pi(\dd x) < \infty.
	\end{equation}
	Then 
	\begin{equation*}
		\lim_{t \to \infty}
		\frac{X^*_{+, 1}(t) - \bone (\gamma \geq 1) t m_1(\lambda) \int_V f_1 \dd \pi}
		{t^{1/\gamma}} 
		= 0 \quad \text{a.s.}
	\end{equation*}
	If \eqref{eq:Cgamma-condition-lemma} holds with $\gamma=2$, then the law of iterated 
	logarithm holds, that is
	\begin{equation*} 
		\limsup_{t \to \infty}  
		\frac{| X^*_{+, 1}(t) - t m_1(\lambda) \int_V f_1 \dd \pi|}
		{\sqrt{2 t \log \log t}} 
		= \left(m_2(\lambda ) \int_V f_1^2 \dd \pi \right)^{1/2}
		\quad \text{a.s.}
	\end{equation*}
\end{lemma}

\begin{proof}
	The process $\{X_{+, 1}^*(t)\}_{t \geq 0}$ is a subordinator with characteristic function
	\begin{equation*}
		\E e^{\ii u X^*_{+, 1}(t)} 
		= \exp \left(t \int_{(0, \infty)} 
		\left(e^{\ii u y} - 1\right) \eta(\dd y) \right),
	\end{equation*} 
	where the L\'evy measure $\eta$ is given by its tail
	\begin{equation*}
		\overline{\eta}(r) 
		= \lambda \times \pi \left(\left\{(z, x) \colon z f_1(x) >r \right\} \right)  
		= \int_V \overline{\lambda }(r/f_1(x)) \pi(\dd x).
	\end{equation*}
	Condition \eqref{eq:Cgamma-condition-lemma} implies that
	\begin{equation} \label{eq:X+1-proof}
		\begin{split}
			 & \gamma \int_1^\infty r^{\gamma - 1} \overline{\eta} (r) \dd r \\ 
			& =\gamma \int_1^\infty r^{\gamma - 1} \int_V 
			\overline{\lambda } (r/f_1(x)) \pi (\dd x) \dd r \\ 
			& = \gamma \int_V f_1(x)^\gamma \int_{1/f_1(x)}^\infty  
			z^{\gamma-1}  \overline{\lambda }(z) \dd z \pi (\dd x)  \\   
			&= \int_V f_1(x)^\gamma \int_{(1/f_1(x), \infty)} 
			\left( z^\gamma -  f_1(x)^{-\gamma} \right) \lambda (\dd z) \pi (\dd x) \\ 
			&= \int_V \int_{(0,\infty)} 
			z^\gamma f_1(x)^\gamma \bone (zf_1(x) > 1) 
			\lambda (\dd z) \pi (\dd x)  - \overline{\eta}(1) < \infty.
		\end{split}		
	\end{equation}
	Hence, $\E X^*_{+, 1} (1)^{\gamma}<\infty$ by \cite[Theorem 25.3]{sato2013}. 
	In the case $\gamma < 1$, since $X^*_{+, 1}(n)$ {is a sum of i.i.d.~increments}, 
	the Marcinkiewicz--Zygmund SLLN (e.g.~\cite[Theorem 6.7.1]{gut2013}) gives 
	$X^*_{+, 1}(n)/n^{1/\gamma} \to 0$ a.s.~and the result follows by monotonicity.
	For $\gamma \geq 1$, the result follows from the process version of 
	Marcinkiewicz--Zygmund SLLN \cite[Theorem 2.1]{tiefeng1993}, since 
	\begin{equation*}
		\E X^*_{+, 1}(1) 
		 = \int_0^\infty \overline{\eta}(r) \dd r 
		 = \int_0^\infty \int_V f_1(x) \overline{\lambda}(z) \pi(\dd x) \dd z 
		 = m_1(\lambda) \int_V f_1 \dd \pi . 
	\end{equation*}	
	In the case $\gamma = 2$, the claim follows from the law of iterated logarithm 
	for L\'evy processes \cite[Proposition 48.9]{sato2013} and the fact that 
	\[
    \Var(X^*_{+, 1}(1)) = \int_{V \times (0, 1] \times (0, \infty)} 
    z^2 f_1(x)^2 \nu(\dd x, \dd s, \dd z) 
	= m_2(\lambda) \int_V f_1^2 \dd \pi.
    \]
\end{proof}

It remains to show that $X^*_{+, 2}(t)$ is negligible in the limit, which is the most challenging part.
We prove this by carefully decomposing the points of the Poisson random measure according to
their contribution to the integral. We then obtain the bounds for 
the number of points in each piece of this decomposition. 

To this end, for $r>0$ and for $0<r_1<r_2$ let
\begin{equation} \label{eq:D-sets-definition}
	\begin{split}	
		& D(r,t)  \coloneqq \{ (x,s, z) \colon 
		z f_2(x,t-s) > r, \, x \in V, \, s \in (0,t],\, z > 0\} \\
		& D(r_1, r_2, t)  \coloneqq \{( x,s, z) \colon  
		r_1 < z f_2(x,t-s) \leq r_2, \, x \in V, \, s \in (0,t], \,  z > 0\}.
	\end{split}
\end{equation}

\begin{lemma} \label{lemma:D-measure-bound}
	For $\gamma \in (1,2]$ assume \eqref{eq:Cgamma-condition-lemma} and $f_1 \in L^\gamma(\pi)$.
	If Assumption~\ref{assum-1} holds, then for every $0 < r < \infty$ and $t > 1$, 
	there exists a constant $C = C(\pi,\lambda,\gamma)$ such that 
	$\nu(D(r, t)) \leq C r^{-1} t^{2-\gamma}$.
\end{lemma}

\begin{proof}
	{ By substituting $u = t-s$, we have
	\begin{align*}
		\nu(D(r,t))
		&= \int_V \int_{(0,\infty)} \int_0^t
		\bone(zf_2(x,u)>r)\dd u\lambda(\dd z)\pi(\dd x) \\
		&= \int_{f_1(x) > t} \int_{(0,\infty)} \int_0^t
		\bone(zf_2(x,u)>r)\dd u\lambda(\dd z)\pi(\dd x)\\
		&\quad +
		\int_{f_1(x) \leq t} \int_{(0,\infty)} \int_0^t
		\bone(zf_2(x,u)>r)\dd u\lambda(\dd z)\pi(\dd x) \\
		& \eqqcolon I_1+I_2.
	\end{align*}
	For $I_1$, using $f_2(x,u)\leq f_1(x)$ and Markov's inequality we obtain
	\begin{align*}
		I_1
		&\leq t\int_{f_1(x) > t} \int_{(0,\infty)}
		\bone(zf_1(x)>r)\lambda(\dd z)\pi(\dd x) \\
		& \leq m_1(\lambda) r^{-1}t
		\int_{f_1(x) > t}f_1(x)\pi(\dd x) \\ 
		&\leq m_1(\lambda) r^{-1}t^{2-\gamma}
		\int_{f_1(x) > t}f_1(x)^\gamma\pi(\dd x) \\
		& \leq m_1(\lambda) r^{-1}t^{2-\gamma}
		\int_V f_1(x)^\gamma\pi(\dd x).
	\end{align*}
	By Assumption~\ref{assum-1}, 
    \[
    f_2(x,u) \leq u^{-1}\int_0^\infty v f(x,v)\dd v
	\leq K u^{-1}f_1(x)^2,
    \]
    so that $zf_2(x,u)>r$ forces
	$u<K r^{-1}z f_1(x)^2$ and
	\begin{align*}
		I_2
		&\leq 
		\int_{f_1(x) \leq t} \int_{(0,\infty)} \int_0^t
		\bone\left(u<K r^{-1}z f_1(x)^2\right)
		\dd u\lambda(\dd z)\pi(\dd x) \\
		& \leq 
		K r^{-1}m_1(\lambda)
		\int_{f_1(x) \leq t}f_1(x)^2\pi(\dd x)\\
		&\leq
		K m_1(\lambda)r^{-1}t^{2-\gamma}
		\int_{f_1(x) \leq t}f_1(x)^\gamma\pi(\dd x) \\
		& \leq
		K m_1(\lambda)r^{-1}t^{2-\gamma}
		\int_V f_1(x)^\gamma\pi(\dd x).
	\end{align*}
	Adding the two gives the claim with $C = (1+K)m_1(\lambda)\int_V f_1^\gamma\dd\pi$.}
\end{proof} 

For a Poisson random variable $N_\ell$ with parameter $\ell>0$,
we will use the tail bounds \cite[Lemma 3.1]{chong-kevei-2020}
\begin{equation} \label{eq:poisson-tail}
	\begin{split}
		&\P(N_\ell \geq n) \leq \frac{\ell^n}{n!}, \quad n \in \N, \\
		&\P(N_\ell \geq x) \leq \exp(-0.19 x), \quad x \geq 2 \ell .
	\end{split}
\end{equation}

\begin{lemma} \label{lemma:finite-variation-3}
	For $\gamma \in (0, 2]$, assume \eqref{eq:Cgamma-condition-lemma}.
	If Assumption~\ref{assum-1} holds, then
	\begin{equation*}
		\lim_{t \to \infty} \frac{X^*_{+, 2}(t)}{t^{1/\gamma}} = 0 \quad \text{a.s.}	
	\end{equation*}
\end{lemma}

\begin{proof}
	If $\gamma < 1$, the statement follows from Lemma \ref{lemma:finite-variation-2} 
	since $X^*_{+, 2}(t) \leq X^*_{+, 1}(t)$. Thus, we assume $\gamma > 1$ and for $n \in \N$ set 
	\begin{equation*}
		Y(n) \coloneqq {\sum_{n-1 \leq \tau_k < n} \zeta_k f_1(\xi_k)}.
	\end{equation*}
	Note that $Y(1), Y(2),\ldots$ form an i.i.d.~sequence. Furthermore, they have the same distribution as $X^*_{+, 1}(1)$.
	Condition \eqref{eq:Cgamma-condition-lemma} implies $\E X^*_{+, 1}(1)^\gamma < \infty$, as was shown in \eqref{eq:X+1-proof}, 
	which in turn yields $\E Y(1)^{\gamma} < \infty$.
	By \cite[Theorem 3.5.1]{embrechts1997}, it follows that 
	\begin{equation} \label{eq:max-con}
		n^{-1/\gamma} \max_{1 \leq i \leq n} Y(i) \to 0 \quad \text{a.s.}
	\end{equation}
	Furthermore, ${\max_{0 < \tau_k < n} \zeta_k f_1(\xi_k)} \leq \max_{1 \leq i \leq n} Y(i)$.
	First we show that
	\begin{equation} \label{eq:X+2-main}
		\lim_{n \to \infty} \frac{X^*_{+,2}(n)}{n^{1/\gamma}} = 0 \quad \text{a.s.}
	\end{equation}
	Let $c_0 = \gamma + 1/\gamma - 2 > 0$, choose $a_0 > 0$ such that 
	$2 - \gamma < a_0 < 2-\gamma + c_0/2$ and set $m = \lfloor 2a_0/c_0 \rfloor$,
	where $\lfloor \cdot \rfloor$ stands for the lower integer part. 
	For $i = 0,1,\ldots,m-1$, define $a_{i+1} = a_i - c_0/2$ and $a_{m+1} = 0$. 
	Finally, select $M \in \N$ such that $M(2-\gamma - a_0) < -1$.
	
	{We begin by handling the large contributions. 
	Recall the sets from \eqref{eq:D-sets-definition}, and for $n \in \N$, define
	\begin{align*}
		&A_{n, 0} \coloneqq \left\lbrace \mu(D(n^{a_0},n) ) > M  \right\rbrace, \\
		&A_{n, i+1} \coloneqq \left\lbrace
		\mu( D( n^{a_{i+1}},  n^{a_i},  n )  ) > {6\log n} + 2Cn^{2-\gamma - a_{i+1}}
		\right\rbrace , \quad i = 0,1,\ldots, m,	
	\end{align*} 
	where $C = C(\pi,\lambda,\gamma)$ is the constant from Lemma \ref{lemma:D-measure-bound}. 
	From \eqref{eq:poisson-tail}, it follows that, for $i = 0,1, \ldots, m$,
	\begin{equation*}
	\begin{split}
     & \P(A_{n, 0}) \leq \frac{\nu ( D(n^{a_0}, n ) )^M}{M!} \leq 
     Cn^{M(2-\gamma - a_0)}, \\
	& \P(A_{n, i+1}) \leq e^{-0.19 (6 \log n + 2 Cn^{2- \gamma - a_{i+1}})}.
	\end{split}
    \end{equation*}
	Since these bounds are summable, the first Borel--Cantelli lemma implies that the events 
	$A_{n, i}$ will occur for only finitely many $n$ almost surely. If we now choose sufficiently large $n$ such that
	$A_{n, 0}$ does not occur, the contribution of the points $(\xi_k, \tau_k, \zeta_k) \in D(n^{a_0}, n)$ is bounded by 
	\begin{equation} \label{eq:A_n0}
		M \max_{0 < \tau_k < n} \zeta_k f_2(\xi_k, n - \tau_k) \leq 
		M \max_{0 < \tau_k < n} \zeta_k f_1(\xi_k) \leq 
		M \max_{1 \leq i \leq n} Y(i),
	\end{equation}
	which is $o(n^{1/\gamma})$ a.s.~by \eqref{eq:max-con}. Similarly, the contribution of the points 
	$(\xi_k,\tau_k,\zeta_k) \in D(n^{a_{i+1}}, n^{a_i}, n)$, if $n$ is large enough such that 
	$A_{n, i+1}$ does not occur, is bounded by 
	\begin{equation} \label{eq:A_ni+1}
		n^{a_i} \left( 6 \log n + 2Cn^{2 - \gamma - a_{i +1} } \right) 
		\leq 6n^{a_0}\log n + 2C n^{1/\gamma - c_0/2},
	\end{equation}	
	which is again $o(n^{1/\gamma})$ since $a_0 < 2 - \gamma + c_0/2 = 1/\gamma - c_0/2 < 1/\gamma$. 
	}
	
	{We now consider the small contributions. For $j = 0,1, \ldots$ and $n \in \N$, define the events 
	\begin{equation*}
		B_{n, j} \coloneqq \left\lbrace \mu(D( 2^{-j-1}, 2^{-j}, n)) > b_{n, j} \right\rbrace	
	\end{equation*}
	with 
    \[
    b_{n,j} \coloneqq 6\log(n(j+1)) + 2\nu\left(D(2^{-j-1},2^{-j},n)\right).
    \]
	Then 
    $2\nu\left(D(2^{-j-1},2^{-j},n)\right) \leq b_{n,j}$,
    so \eqref{eq:poisson-tail} applies and gives
	\begin{equation*}
		\P(B_{n,j})
		\leq \exp(-0.19 b_{n,j})
		\leq \exp(-0.19\cdot 6\log(n(j+1)))
		= (n(j+1))^{-1.14}.
	\end{equation*}
	By the first Borel--Cantelli lemma, $B_{n,j}$ occurs for at most finitely
	many $n$ and $j$. Hence, for all sufficiently large $n$, the sum of the
	contributions such that $\zeta_k f_2(\xi_k,n-\tau_k) \in(2^{-j-1},2^{-j}]$, for $j=0,1,\ldots$ is bounded by
	\begin{equation*}
		\sum_{j \geq 0} 2^{-j}b_{n,j}
		= 6\sum_{j \geq 0} 2^{-j}\log(n(j+1)) 
		+ 2\sum_{j \geq 0} 2^{-j}\nu\left(D(2^{-j-1},2^{-j},n)\right).
	\end{equation*}
	The first term is at most $c\log n$, since
	\[
    \sum_{j\geq0}2^{-j}\log(n(j+1)) = \log n\sum_{j\geq0}2^{-j}
	+ \sum_{j\geq0}2^{-j}\log(j+1).
    \]
	To estimate the second term, note that on
	$D(2^{-j-1},2^{-j},n)$ we have $2^{-j} < 2z f_2(x,n-s)$. Therefore,
	\begin{align*}
		\sum_{j \geq 0}
		2^{-j}\nu\left(D(2^{-j-1},2^{-j},n)\right)
		&\leq
		2\sum_{j \geq 0}
		\int_{D(2^{-j-1},2^{-j},n)}
		z f_2(x,n-s)\nu(\dd x,\dd s,\dd z)\\
		&\leq
		2m_1(\lambda)
		\int_V\int_0^\infty
		(n\wedge v)f(x,v)\dd v\pi(\dd x),
	\end{align*}
	and by Assumption~\ref{assum-1}, as in the proof of
	Lemma~\ref{lemma:finite-variation-1},
	$\int_0^\infty (n\wedge v)f(x,v)\dd v \leq K^{\gamma-1}n^{2-\gamma}f_1(x)^\gamma$.
	Consequently,
	\begin{equation} \label{eq:sum-bnj}
	\begin{split}	
    \sum_{j=0}^\infty 2^{-j}b_{n,j}
		& \leq c\log n
		+ 4m_1(\lambda) K^{\gamma - 1} n^{2-\gamma} \int_V f_1(x)^\gamma\pi(\dd x) \\
		& \leq c\log n+cn^{2-\gamma}.
	\end{split}
    \end{equation} }

	{Combining \eqref{eq:A_n0}, \eqref{eq:A_ni+1}, and \eqref{eq:sum-bnj}, 
	we conclude that for all sufficiently large $n$,
	\begin{equation*}
		X^*_{+,2}(n)
		\leq M \max_{1 \leq i \leq n} Y(i)
			+ 6n^{a_0} \log n  
			+ 2C n^{1/\gamma-c_0/2}
			+ c\log n
			+ cn^{2-\gamma}.
	\end{equation*}} 
	All terms are asymptotically smaller than $n^{1/\gamma}$. Specifically, for the {first term}
	this is stated in \eqref{eq:max-con}. Therefore, \eqref{eq:X+2-main} follows. Now, for $t \in (n, n+1)$
	\begin{align*}
		X_{+, 2}^*(t) &=  
		\int_{V \times (0, t] \times (0, \infty)} zf_2(x, t-s) 
		\mu(\dd x, \dd s, \dd z) \\ 
		&\leq 
		\int_{V \times (0, n] \times  (0, \infty)} z { f_2(x, n-s) } 
		\mu(\dd x, \dd s, \dd z) \\
        & \quad 
		+ \int_{V \times (n, n+1] \times (0, \infty)} z f_1(x)
		\mu(\dd x, \dd s, \dd z) \\ 
		&= X^*_{+, 2}(n) + Y(n + 1),		
	\end{align*}	
	which implies $t^{-1/\gamma}X^*_{+, 2}(t) \leq n^{-1/\gamma} (X^*_{+, 2}(n) + Y(n+1)) \to 0$ \text{a.s.}
\end{proof}

\begin{proof}[Proof of Theorem~\ref{thm:finite-variation}]
	{
	If $\lambda((-\infty, 0)) = 0$, we center each of the three terms in
	\eqref{eq:decomposition-finite-variation} separately. The middle one is
	covered by Lemma~\ref{lemma:finite-variation-2}. For the other two, Fubini's theorem
	gives
	\[
    \E X^*_-(t) = \E X^*_{+,2}(t)
	= m_1(\lambda) \int_V \int_0^t f_2(x,u) \dd u \pi(\dd x),
    \]
	so that the two expectations cancel, leaving
	$t^{-1/\gamma} (X^*_-(t) - X^*_{+,2}(t))$, which is covered by
	Lemmas~\ref{lemma:finite-variation-1} and \ref{lemma:finite-variation-3}.}
	
	Otherwise, we apply the same decomposition,
	but treating positive and negative jumps separately. More precisely, by considering 
	the measures $\mu \bone(z > 0)$ and $\mu \bone(z \leq 0)$, we get 
	\[
    X^*(t) = X^{*, +}_-(t) - X^{*, -}_-(t)  + 
		X^{*, +}_{+, 1}(t) - X^{*, -}_{+, 1}(t)
		- X^{*, +}_{+, 2}(t) + X^{*, -}_{+, 2}(t).
    \]
	The statement for $\gamma \in (0, 2)$ follows by applying Lemmas~\ref{lemma:finite-variation-1}, \ref{lemma:finite-variation-2} and
	\ref{lemma:finite-variation-3} to positive and negative jumps separately.
	
	If $\gamma = 2$, then the proof follows in the same way as for $\gamma \in (0, 2)$. 
	In general, we decompose $X^*_-(t)$ and $X_{+, 2}^*(t)$ as in the preceding case. 
	After normalization by $\sqrt{2t \log \log t}$, each of these terms converges to $0$ almost surely.
	The claim then follows from \cite[Proposition 48.9]{sato2013}, applied to $X_{+, 1}^*(t)$.	
\end{proof}

\subsection{Proof of Theorem~\ref{thm:infinite-variation}}

Suppose that $\int_{|z|\leq 1} |z| \lambda(\dd z) = \infty$, $b = 0$ and $a = 0$. 
In this case, centering is required in \eqref{eq:levy-ito}. We will first prove the claim in the case when
$\lambda $ is supported on $(0, 1]$. Then we combine this with Theorem~\ref{thm:finite-variation} 
to establish Theorem~\ref{thm:infinite-variation}. The integrated process takes the form
\begin{align*}
	X^*(t) &= 
	\int_{V \times (-\infty, 0] \times (0, 1]} 
	z \left(f_2(x, -s) - f_2(x, t-s) \right)
	(\mu - \nu )(\dd x, \dd s, \dd z) \\ 
	& \quad + 
	\int_{V \times (0, t] \times (0, 1]} 
	z \left(f_1(x) - f_2(x, t-s)\right)
	(\mu - \nu )(\dd x, \dd s, \dd z) \\
 & \eqqcolon X_-^*(t) + X_+^*(t).	
\end{align*}

The following lemmas ensure that the paths are sufficiently smooth, so that it suffices to prove the almost sure limits along integer times. {The first one is \cite[Lemma 15]{grahovac2025}, however we restate it for completeness.} 

{
\begin{lemma} \label{lemma:supY}
    Assume that for a process $\{Y(t), \, t \in \R \}$
    \begin{equation*}
        \sup_{n\in\mathbb{N}} \mathbf{E}\left( \sup_{t\in[n,n+1]} |Y(t)-Y(n)| \right)^{2} < \infty,        
    \end{equation*}
    and for an increasing function $a(t)$ satisfying $\lim_{n\to\infty} a(n)/a(n+1) = 1$ and $\sum_{n} a(n)^{-2} < \infty$, we have $\limsup_{n\to\infty} |Y(n)|/a(n) \leq c$ a.s., with some $c \geq 0$. Then $\limsup_{t\to\infty} |Y(t)|/a(t) \leq c$ a.s. 
\end{lemma}
}

\begin{lemma} \label{lemma:Delta-t}
	If Assumption~\ref{assum-2} holds, then there exist continuous modifications 
	of $X^*_-$ and $X^*_+$ such that for $\theta \in [0, 1/2)$ 
	\begin{equation*}
		\sup_{n \in \N} \E \left(
		\sup_{t_1 \ne t_2, \, t_1, t_2 \in [n, n+1]}
		\frac{|X^*_-(t_2) - X^*_-(t_1)| + |X^*_+(t_2) - X^*_+(t_1)|}{|t_2 - t_1|^{\theta}} 		
		\right)^2 < \infty.		
	\end{equation*}
\end{lemma}

\begin{proof}
	For $n \in \N$, let $t_1, t_2 \in [n, n+1]$ be such that $t_1 < t_2$.
	We treat each increment separately, applying \cite[Theorem 1]{marinelli2014} 
	with $p = \alpha = 2$ to bound the second moment. The increments are given by
	\begin{align*}
		& X_-^*(t_2) - X_-^*(t_1) \\
		&= \int_{V \times (- \infty, 0] \times (0, 1]}
		z \left(f_2(x, t_1 - s) - f_2(x, t_2 - s) \right)
		(\mu - \nu )(\dd x, \dd s, \dd z),\\
		& X^*_+(t_2) - X^*_+(t_1) \\ 
		&= \int_{V \times (0, t_1] \times (0, 1]} 
		z (f_2(x, t_1- s) - f_2(x, t_2- s))
		(\mu - \nu )(\dd x, \dd s, \dd z)  \\
		& \quad + \int_{V \times (t_1, t_2] \times (0, 1]} 
		z (f_1(x) - f_2(x, t_2- s))
		(\mu - \nu )(\dd x, \dd s, \dd z). 
	\end{align*}	
	By \eqref{eq:f-cond-int1} of Assumption~\ref{assum-2}, we have the following for $X^*_-$
	\begin{align*}
		& \E (X_-^*(t_2) - X_-^*(t_1))^2 \\
		& \leq c \int_{-\infty}^0 \int_V \int_{(0, 1]} z^2  
		\left(f_2(x, t_1 - s) - f_2(x, t_2 - s) \right)^2
		\lambda (\dd z) \pi (\dd x) \dd s \\
		& \leq c \, c_2 \, m_2(\lambda )  \int_V (\Delta t \wedge f_1(x))^2 f_1(x)
		\pi (\dd x) \\
        & \leq c \, c_2 \, (\Delta t)^2 \, m_2(\lambda )  
		\int_V  f_1 \dd \pi,
	\end{align*}	
	where $\Delta t = t_2 - t_1$.
	Using also \eqref{eq:f-cond-int2}, we obtain
	\begin{align*}
		& \E(X^*_+(t_2) - X^*_+(t_1))^2 \\ 
		& \leq c m_2(\lambda ) \int_V  \left(
		\int_0^{t_1} (f_2(x,t_1 - s) - f_2(x, t_2-s) )^2 \dd s \right. \\
		& \quad \left. +
		\int_{t_1}^{t_2} (f_1(x) - f_2(x,t_2 -s))^2 \dd s \right) \pi (\dd x) \\
		& \leq 2 c_2  m_2(\lambda ) \int_V \left(
		(\Delta t)^2 f_1(x) + (\Delta t)^2 f_1(x)
		\right) \pi (\dd x) \\
        & \leq 4 c_2 (\Delta t)^2 m_2(\lambda ) \int_V f_1 \dd \pi.
	\end{align*}
	The previous bounds hold for any given $t_1, t_2 \in [n, n +1]$, implying they are uniform in $n$. 
	The claim follows by taking $N = 1$, $T = [n, n+1]$, $\gamma = 2$, $p = 2$, 
	and $\theta < 1/2$ in \cite[Theorem 4.3]{khoshnevisan2009}. 
\end{proof}

For $\kappa \in (0, 1]$, define 
\begin{equation} \label{eq:S-sets}
	\begin{split}
		& A(t; \kappa) \\
        & = \left\{
		(x, s, z) \colon z (f_1(x) - f_2(x, t- s)) \leq t^\kappa
		, \, s \in (0, t], \, x \in V, \, z \in (0,1] \right\}, \\
		& B(t; \kappa) \\
		& = \left\{
		(x, s, z) \colon z (f_2(x, -s) - f_2(x, t - s)) \leq t^\kappa 	
		, \, s \leq 0, \, x \in V, \, z \in (0, 1] \right\}.
	\end{split}
\end{equation}
These sets are used to define the truncated processes, namely 
\begin{align*}
	&X^*_{+, <, \kappa} (t) \coloneqq  \int_{A(t; \kappa)}
	z (f_1(x) - f_2(x, t - s)) (\mu - \nu )(\dd x, \dd s, \dd z), \\
	&X^*_{-, <, \kappa} (t) \coloneqq  \int_{B(t; \kappa)} 
	z (f_2(x, -s) - f_2(x, t - s)) (\mu - \nu )(\dd x, \dd s, \dd z),\\
	&X^*_{+, >, \kappa} (t) \coloneqq  \int_{A(t; \kappa)^c}
	z (f_1(x) - f_2(x, t - s)) (\mu - \nu )(\dd x, \dd s, \dd z), \\
	&X^*_{-, >, \kappa} (t) \coloneqq  \int_{B(t; \kappa)^c} 
	z (f_2(x, -s) - f_2(x, t - s)) (\mu - \nu )(\dd x, \dd s, \dd z),
\end{align*}
where the complements are taken with respect to $V \times (0, t] \times (0,1]$ and 
$V \times (-\infty, 0] \times (0,1]$, respectively. Note that all of the above processes 
depend on the choice of $\kappa$. For notational convenience, we will occasionally write 
$X^*_{\pm, <, \kappa}$ and $X^*_{\pm, >, \kappa}$ to refer to both processes simultaneously. 
In the proofs, we suppress the lower index $\kappa$ from the notation. 

Recall the constants $\alpha$, $\beta$ introduced before Theorem~\ref{thm:infinite-variation}. 
In the following, we will rely on the inequalities involving $\alpha$ and $\beta$:
\begin{align*}
	&\int_{(\tau, 1]} z^u \lambda (\dd z) \leq \tau^{u - \beta} m_{\beta}(\lambda ), 
	\quad \tau \in (0,1], u \leq \beta, \\
	&\int_{(0, \tau]} z^u \lambda (\dd z)  \leq \tau^{u - \beta} m_{\beta}(\lambda ), 
	\quad \tau \in (0,1],  u \geq \beta, \\
	&\int_{ f_1^{-1}((1,T])} f_1^v \, \dd \pi  \leq 
	T^{v-1-\alpha} \int_{f_1^{-1} ((1,\infty)) } f_1^{1+\alpha} \dd \pi,
	\ T > 1,  v \geq 1 + \alpha, \\
	&\int_{ f_1^{-1} ((T,\infty))} f_1^v \dd \pi \leq 
	T^{v - 1 - \alpha} \int_{f_1^{-1} ((1,\infty))} f_1^{1+\alpha} \dd \pi, 
	\ T > 1, v \leq 1 + \alpha.
\end{align*}

\begin{lemma} \label{lemma:log-theta}
	If Assumption~\ref{assum-2} holds, then for $|\theta| \leq 4t^{-\kappa}$
	\begin{equation*} 
		\log \E e^{\theta X^*_{-, <, \kappa}(t)} \leq
		\begin{cases}
			c \theta^2 t^{\kappa(2 - \beta) + \beta - \alpha}, & \alpha \leq \beta, \\
			c \theta^2 t^{\kappa(2-\alpha)_+}, & \alpha \geq \beta,		
		\end{cases}
	\end{equation*}
	and 
	\begin{equation*}
		\log \E e^{\theta X^*_{+, <, \kappa}(t)} \leq 
		\begin{cases}
			c \theta^2 t^{\kappa(2- \beta) + \beta - \alpha}, & \alpha \leq \beta -1, \\
			c \theta^2 t^{1 + \kappa(1- \alpha)_+}, & \alpha \geq \beta -1.
		\end{cases}
	\end{equation*}
\end{lemma}

\begin{proof}
	The moment generating functions are 
    \[
    \E e^{\theta X^*_{\pm, <}(t)} = \exp \left(
	\int (e^{\theta y} - 1 - \theta y ) \Pi_t^{\pm} (\dd y)
	\right),
    \]
    with the L\'evy measures given by
	\begin{align*}
		& \Pi_t^+(B) \coloneqq  \nu \left( 
		\{(x, s, z) \in A(t; \kappa) \colon
		z \left(f_1(x) - f_2(x, t-s) \right) \in B \}
		\right), \\
		& \Pi_t^-(B) \coloneqq  \nu \left( 
		\{(x, s, z) \in B(t; \kappa) \colon  
		z \left(f_2(x, -s) - f_2(x, t-s)\right) \in B \}
		\right), 		
	\end{align*}
	for $B \in \mathcal{B}(\R)$. Clearly, $\Pi_t^{\pm} ((t^{\kappa}, \infty)) = 0$. 
	For $|u| \leq 4$, we have the elementary inequality $e^{u} - 1 - u \leq 4u^2$, 
	from which it follows that 
    \[
    \log \E e^{\theta X^*_{\pm, <}(t)} \leq 
	4 \theta^2 \int y^2 \Pi_t^{\pm} (\dd y),
    \]
    for $|\theta| \leq 4t^{-\kappa}$.
	
	We start with $X^*_{-, <}(t)$. If $\beta \geq \alpha$, Assumption~\ref{assum-2} gives
	\begin{align*}
		& \log \E e^{\theta X^*_{-, <}(t)} \\
		&\leq 4\theta^2 \int_{-\infty}^0 \int_V \int_{(0, 1]} 
		z^2 \left(f_2(x, -s) - f_2(x, t-s)\right)^2 
		\bone(B(t; \kappa)) 
		\lambda (\dd z) \pi (\dd x) \dd s \\ 
		&\leq  4 \theta^2 m_{\beta}(\lambda ) t^{\kappa(2- \beta)} 
		\int_V \int_{- \infty}^0 \left(f_2(x, -s) - f_2(x, t-s)\right)^\beta 
		\dd s \pi (\dd x) \\ 
		&\leq 4 \theta^2 m_{\beta}(\lambda ) t^{\kappa(2- \beta)} 
		\int_V c_\beta (t \wedge f_1(x))^\beta f_1(x) \pi (\dd x) \\ 
		& \leq 4 \theta^2 m_{\beta}(\lambda ) c_\beta t^{\kappa(2- \beta) + \beta - \alpha}
		\int_V  f_1^{1 + \alpha} \dd \pi .  
	\end{align*}
	For $\beta \leq \alpha \leq 2$, we have
	\begin{align*}
		& \log \E e^{\theta X^*_{-, <}(t)} \\ 
		&\leq 4 \theta^2 m_{\alpha}(\lambda ) t^{\kappa(2- \alpha)} 
		\int_V \int_{- \infty}^0 \left(f_2(x, -s) - f_2(x, t-s)\right)^\alpha 
		\dd s \pi (\dd x) \\ 
		&\leq 4 \theta^2 m_{\alpha}(\lambda )  c_\alpha t^{\kappa(2- \alpha)} 
		\int_V f_1^{1 + \alpha} \dd \pi.  
	\end{align*}
	In the case $\alpha \geq 2$, the bound is straightforward
	\begin{equation*}
    \begin{split}
		\log \E e^{\theta X^*_{-, <}(t)}  
		& \leq  4 \theta^2 m_{2}(\lambda ) 
		\int_V \int_{- \infty}^0 \left(f_2(x, -s) - f_2(x, t-s)\right)^2 
		\dd s \pi (\dd x) \\
		& \leq 4  \theta^2  m_{2}(\lambda ) c_2 
		\int_V f_1^{3} \dd \pi.
	\end{split}
    \end{equation*}
	
	For $X^*_{+, <}(t)$, we have 
	\begin{equation*}
	\begin{split}	  
	& \log \E e^{\theta X^*_{+, <}(t)} \\
    & \leq 
		4\theta^2 \int_0^t \int_V \int_{(0, 1]} 
		z^2 \left(f_1(x) - f_2(x, t-s)\right)^2 
		\bone(A(t; \kappa)) 
		\lambda (\dd z) \pi (\dd x) \dd s.
	\end{split}	
    \end{equation*}
	We again divide the analysis into separate cases. 
	If $\beta \leq 1+\alpha$ and $\alpha \leq 1$, we have by \eqref{eq:f-cond-int2}
	\begin{align*}
		\log \E e^{\theta X^*_{+, <}(t)} 
		&\leq 4 \theta^2 m_{1+\alpha}(\lambda ) t^{\kappa(1-\alpha)} \int_V \int_0^t 
		\left(f_1(x) - f_2(x, t-s)\right)^{1+\alpha} \dd s \pi (\dd x)  \\ 
		&\leq 4 \theta^2 m_{1+\alpha}(\lambda ) c_{1+\alpha} t^{\kappa(1-\alpha) + 1} 
		\int_V f_1^{1+\alpha}  \dd \pi, 
	\end{align*}
	while for $\alpha \geq 1$
	\begin{equation*}
		\log \E e^{\theta X^*_{+, <}(t)} \leq 
		4 \theta^2 m_{2}(\lambda ) c_2
		t \int_V f_1^2 \dd \pi. 
	\end{equation*}	
	Finally, if $ \alpha \leq \beta - 1$ and $\alpha \leq  1$, Assumption~\ref{assum-2} gives
	\begin{align*}
		\log \E e^{\theta X^*_{+, <}(t)} &\leq  
		4 \theta^2 	m_{\beta}(\lambda ) t^{\kappa(2-\beta)} 
		\int_V \int_0^t (f_1(x) - f_2(x, t-s))^\beta 
		\dd s \pi (\dd x) \\ 
		& \leq 4 \theta^2 m_{\beta}(\lambda ) c_\beta t^{\kappa(2-\beta)} 
		\int_V t^{\beta - \alpha} f_1^{\alpha + 1} \dd \pi.
	\end{align*} 
\end{proof}

The previous bound combined with Markov's inequality and first Borel--Cantelli lemma imply the following statement.

\begin{cor} \label{cor:kappa-bound}
	Let $\kappa \in [1/2, 1]$ be such that
	\begin{equation*}
		\kappa \geq
		\begin{cases}
			1 - \alpha/\beta, &  \alpha \leq \beta -1, \\ 
			1/(1+\alpha), & \beta -1 \leq \alpha \leq 1, \\
			1/2, & \alpha \geq 1.		 	
		\end{cases}
	\end{equation*}	
	If Assumption~\ref{assum-2} holds, then
	{\begin{equation*}
		\limsup_{n \to \infty} \frac{|X^*_{\pm, <, \kappa}(n)|}{n^{\kappa}\log n} 
		\leq \frac{1}{4} \quad \text{a.s.}
	\end{equation*}}
\end{cor}

\begin{lemma} \label{lemma:mu-complements}
	If for some $\gamma \geq 1$ 
	\begin{equation*}
		\int_V \int_{(0, 1]} 
		z^\gamma f_1(x)^\gamma \bone(z f_1(x) > 1) 
		\lambda (\dd z) \pi (\dd x) < \infty,
	\end{equation*}
	then, for all large enough $t$, $\mu(A (t; 1/\gamma)^c) = 0$ and 
	$\mu(B(t;1/\gamma)^c) = 0$ a.s., for $A$ and $B$ defined in \eqref{eq:S-sets}.
\end{lemma}

\begin{proof}
	Define 
	\begin{align*}
		S_{-}^n &\coloneqq \{(x, s, z) \colon z(f_2(x, -s) - f_2(x, n-s)) > n^{1/\gamma}, 
		\, s \in (-n-1, -n]\}, \\
		S_{+}^n &\coloneqq \{(x, s, z) \colon z(f_1(x) - f_2(x, n-s)) > n^{1/\gamma}, 
		\, s \in  (n, n+1]\}.
	\end{align*}
	Then
	\begin{align*}
		\nu (S_-^n) &= \int_{-n-1}^{-n} \int_V \int_{(0, 1]} 
		\bone(z(f_2(x, -s) - f_2(x, n-s)) > n^{1/\gamma}) 
		\lambda (\dd z) \pi (\dd x) \dd s	\\ 
		&\leq \int_V \int_{(0, 1]} \bone(z f_1(x) > n^{1/\gamma}) \lambda (\dd z)\pi (\dd x ),
	\end{align*}
	and an identical bound holds for $\nu (S_+^n)$. Therefore,
	\begin{align*}
		\sum_{n = 1}^{\infty} \nu (S_{\pm}^n) &\leq
		\int_V \int_{(0, 1]} \sum_{n = 1}^{\infty} 
		\bone(n < (zf_1(x))^\gamma) \lambda (\dd z) \pi (\dd x) \\ 
		&\leq \int_V \int_{(0, 1]} (zf_1(x))^\gamma \bone(zf_1(x) > 1) \lambda (\dd z) \pi (\dd x) < \infty.
	\end{align*}
	By the first Borel--Cantelli lemma, $\mu(S_{\pm}^n) = 0$ 
	almost surely for all large enough $n$, and the result follows.	
\end{proof}

\begin{lemma} \label{lemma:compensator}
	Let $\kappa \in (0,1)$ and suppose that Assumption~\ref{assum-2} holds. Then, 
	for the positive jumps, we have
	\begin{equation} \label{eq:compensator-plus}
	\begin{split}
        & \int_{A(t; \kappa)^c} z(f_1(x) - f_2(x, t-s)) \nu (\dd x, \dd s, \dd z) \\
		& \leq 
		\begin{cases}
			c t^{\beta - \alpha + \kappa(1 -\beta)}, & \alpha \leq \beta - 1, \\
			c t^{1 - \kappa\alpha}, & \alpha \geq \beta - 1, 
		\end{cases}
    \end{split}
	\end{equation}
	and 
	\begin{equation} \label{eq:lambda-plus}
		\nu (A(t; \kappa)^c) \leq 
		\begin{cases}
			ct^{\beta - \alpha - \kappa \beta}, & \alpha \leq \beta - 1, \\ 
			ct^{1-\kappa(1+\alpha)}, & \alpha \geq \beta - 1.
		\end{cases}
	\end{equation}	
	For the negative jumps, we have
	\begin{equation} \label{eq:compensator-minus}
	\begin{split}	
    & \int_{B(t; \kappa)^c} z (f_2(x, -s) - f_2(x, t-s))
	\nu (\dd x, \dd s, \dd z) \\
	& \leq 
	\begin{cases}
	c t^{\beta - \alpha + \kappa(1-\beta)}, & \alpha \leq \beta, \\
	c t^{\kappa(1-\alpha)}, & \alpha \geq \beta,
	  \end{cases}
    \end{split}
	\end{equation}
	and 
	\begin{equation} \label{eq:lambda-minus}
		\nu (B(t; \kappa)^c) \leq  
		\begin{cases}
			ct^{\beta - \alpha - \kappa \beta}, & \alpha \leq \beta, \\
			ct^{-\kappa \alpha}, & \alpha \geq \beta.
		\end{cases}
	\end{equation}
\end{lemma}

\begin{proof}
	The proof follows an approach analogous to that employed in the proof of Lemma~\ref{lemma:log-theta}. 
	In fact, many of the key steps were already made there and will be directly used here. 
	Let $A_-$ and $A_+$ denote the left-hand sides of \eqref{eq:compensator-minus} and \eqref{eq:compensator-plus},
	respectively. For $\alpha \leq \beta$, $A_-$ can be bounded as follows
	\begin{align*}
		A_- &\leq m_{\beta}(\lambda ) t^{\kappa(1-\beta)} 
		\int_V  \int_{-\infty}^0 
		(f_2(x, -s) - f_2(x, t-s))^\beta \dd s \pi (\dd x) \\
		& \leq m_\beta(\lambda) c_\beta t^{\kappa(1-\beta)} 
		\int_V (t \wedge f_1(x))^\beta f_1(x) \pi(\dd x) \\
        & \leq c t^{\kappa(1-\beta) + (\beta - \alpha)} \int_V f_1^{1+\alpha} \dd \pi.
	\end{align*}
	Similarly, in the case $\alpha \geq \beta$, we have 
	\begin{equation*}
		A_- \leq m_{\alpha}(\lambda ) t^{\kappa(1-\alpha)} 
		\int_V  \int_{-\infty}^0 
		(f_2(x, -s) - f_2(x, t-s))^\alpha \dd s \pi (\dd x) 
		\leq c t^{\kappa(1-\alpha)}.
	\end{equation*}
	If $\beta \leq 1 + \alpha$, then by \eqref{eq:f-cond-int2}
	\begin{equation*}
		A_+ \leq m_{1+\alpha}(\lambda ) t^{-\kappa \alpha} 
		\int_V \int_0^t (f_1(x) - f_2(x, t-s))^{1+\alpha} \dd s \pi (\dd x) 
		\leq c t^{1- \kappa \alpha}.
	\end{equation*}
	Finally, for $\beta > 1 +\alpha$, we have 
	\begin{equation*}
		A_+  \leq m_{\beta}(\lambda ) t^{-\kappa (\beta-1)} \int_V \int_0^t 
		(f_1(x) - f_2(x, t-s))^{\beta} \dd s \pi (\dd x) 
		\leq c t^{\beta - \alpha - \kappa (\beta-1)}.
	\end{equation*}
	{The bounds \eqref{eq:lambda-plus} and \eqref{eq:lambda-minus} follow} from 
	\eqref{eq:compensator-plus} and \eqref{eq:compensator-minus} combined with a 
	Markov inequality.
\end{proof}

\begin{remark}
	We note that the bounds in \cite[Lemma 19]{grahovac2025} could be improved as 
	in the previous lemma. However, the weaker version does not compromise the proofs there.
\end{remark}

\begin{lemma} \label{lemma:contradiction}
	Suppose that $\beta \geq 1 + \alpha$,  and that Assumption~\ref{assum-2} holds. Then, 
	for any $\kappa > 1 - \tfrac{\alpha}{\beta}$, for all large enough $t$, 
	$\mu(A(t; \kappa)^c ) = 0$ and $\mu(B(t;\kappa)^c) = 0$ almost surely.
\end{lemma}

\begin{proof}
	Choose $\kappa' \in (1-\alpha/\beta, \kappa)$. Then, by \eqref{eq:lambda-plus}, 
	\[
    \nu (A(t; \kappa')^c) \leq ct^{\beta - \alpha - \kappa' \beta} \leq ct^{-\delta}
    \]
	for some $\delta = \delta(\kappa') > 0$. Now, choose $d >0$ such that $d\delta > 1$. 
	By the Borel--Cantelli lemma, it follows that for all large enough $n$, $\mu(A(n^d; \kappa')^c) = 0$ 
	almost surely. Assume that $\mu(A(t; \kappa)^c) \geq 1$ for some $t \in ((n-1)^d, n^d)$. 
	In other words, there exists a point $(\xi_k, \tau_k, \zeta_k) \in A(t; \kappa)^c$ such that for some $n$ we have
	\begin{equation*}
		(n-1)^{d\kappa} < t^\kappa < \zeta_k(f_1(\xi_k) - f_2(\xi_k, t-\tau_k)) 
		\leq \zeta_k(f_1(\xi_k) - f_2(\xi_k, n^d - \tau_k)) < n^{d\kappa'}.
	\end{equation*}	
	This leads to  $(1 - \tfrac{1}{n})^{d\kappa} < n^{d(\kappa' - \kappa)}$,
	which is a contradiction since the left-hand side converges to $1$, 
	whereas the right-hand side converges to $0$. For $\mu(B(t; \kappa)^c)$, we use 
	\eqref{eq:lambda-minus}, and the proof follows by the same argument. 
\end{proof}

Combining the previous results, we now prove the statement of Theorem~\ref{thm:infinite-variation} 
under the assumption that $\lambda$ is supported on $[-1, 1]$.

\begin{cor} \label{cor:small-jumps}
	Assume that $\int_{|z| \leq 1} |z| \lambda (\dd z) = \infty$, $a = 0$, $b = 0$ and that 
	$\lambda $ is supported on $[-1, 1]$. Moreover, suppose that Assumption~\ref{assum-2} holds. Then
	\begin{equation*}
		\limsup_{t \to \infty} \frac{|X^*(t)|}{t \log t} \leq 1 \quad \text{a.s.}
	\end{equation*} 
	If $\beta \leq 1 + \alpha$ and there exists $\gamma \in [1,2]$ such that 
	\begin{equation} \label{eq:Cgamma-condition-infinite-variation}
		\int_V \int_{|z| \leq 1} |z|^{\gamma} f_1(x)^{\gamma} \bone(|z| f_1(x) > 1) 
		\lambda (\dd z) \pi (\dd x)  < \infty, 
	\end{equation}
	then 
	\begin{equation*}
		\limsup_{t\to \infty} \frac{|X^*(t)|}{t^{1/\gamma} \log t} \leq 1 \quad \text{a.s.} 
	\end{equation*}
	If $\beta > 1 + \alpha$, then for any $\frac{1}{\gamma}>1-\frac{\alpha}{\beta}$
	\begin{equation*}
		\lim_{t\to \infty} \frac{X^*(t)}{t^{1/\gamma}} = 0 \quad \text{a.s.}
	\end{equation*}
\end{cor}

\begin{proof}
    {It is enough to prove the claim along integer times, since the extension to continuous time follows from Lemmas~\ref{lemma:supY} and \ref{lemma:Delta-t}.} Assume $\lambda $ is supported on $(0,1]$. The first claim follows from 
	Corollary~\ref{cor:kappa-bound} with $\kappa = 1$, since 
	\begin{equation*}
		{\limsup_{n\to \infty} \frac{|X^*(n)|}{n \log n} \leq 
		\limsup_{n\to \infty} \frac{|X_-^*(n)| + |X_+^*(n)|}{n \log n}  
		\leq  \frac{1}{2} \quad \text{a.s.}}
	\end{equation*}
	If $\beta -1 \leq  \alpha$ and \eqref{eq:Cgamma-condition-infinite-variation} holds, 
	then $\gamma \leq 1 + \alpha$. By Lemma~\ref{lemma:mu-complements} the terms $X^*_{\pm, >, 1/\gamma}(n)$ 
	are negligible in the limit, while Lemma~\ref{lemma:compensator} with $\kappa = 1/\gamma$ ensures 
	the same for the compensating terms. Corollary~\ref{cor:kappa-bound} with $\kappa = 1/\gamma$ gives the desired result. For $\beta -1 \geq \alpha$, we use a similar approach. The claim follows from Lemmas~\ref{lemma:contradiction}, \ref{lemma:compensator} and Corollary~\ref{cor:kappa-bound} with $\kappa > 1 - \alpha/\beta$. If $\lambda $ is supported on $[-1,1]$, the result can be applied separately to $\mu \bone(z > 0)$ and $\mu \bone(z < 0)$.    
\end{proof}

\begin{proof}[Proof of Theorem~\ref{thm:infinite-variation}]
	We decompose $\mu$ into $\mu \bone(|z| \leq 1)$ and $\mu \bone(|z| > 1)$ so that 
	the integrated process can be written as 
    \[
    X^*(t) = X_-^{*, \leq}(t) + X_-^{*, >}(t) + X_+^{*, \leq}(t) + X_+^{*, >}(t).
    \]
    The statement follows from Corollary~\ref{cor:small-jumps} applied to $X_-^{*, \leq}(t)$ and $X_+^{*, \leq}(t)$, and Theorem~\ref{thm:finite-variation} to $X_-^{*, >}(t)$ and $X_+^{*, >}(t)$.
\end{proof}

\subsection{Proof of Theorem \ref{tm-gaussian}}

Since $\Lambda$ is a Gaussian random measure, we have 
$\E e^{\ii \theta \Lambda(A)} = 
e^{-b\frac{\theta^2}{2}(\pi  \times \Leb)(A)}$, and for the integrated process,  \cite[Proposition 2.6]{rajput1989} gives
	\begin{equation*}
		\log \E e^{\ii \theta X^*(t)} = - b \frac{\theta^2}{2}
		\int_V \int_{\R} \left(\int_0^t f(x, u-s) \bone(s \leq u) \dd u \right)^2   
		\dd s \pi (\dd x).
	\end{equation*}
	Therefore, $X^*$ is a Gaussian process with mean $0$ and variance given by
	\begin{equation*}
		Q(t) \coloneqq \frac{1}{2} \Var(X^*(t)) = \frac{b}{2}
		\int_V \int_{\R} \left(\int_0^t f(x, u-s) \bone(s \leq u) \dd u \right)^2   
		\dd s \pi (\dd x).		
	\end{equation*}
	Clearly, $Q$ is non-decreasing and $Q(0) = 0$. By the decomposition of $X^*(t)$, it follows that
	\begin{align*}
		Q(t) &=  \frac{b}{2} \int_V \int_{\R} \Big(
		(f_2(x, -s) - f_2(x, t-s)) \bone(s < 0)  \\
        & \hspace{2.1cm} 
        +(f_1(x) - f_2(x, t-s)) \bone(0 \leq s \leq t) \Big)^2 \dd s\pi (\dd x)	\\ 
		&= \frac{b}{2} \int_V \Big( 
		\int_{-\infty}^0 \left(f_2(x, -s) - f_2(x, t-s) \right)^2 \dd s \\
        & \hspace{2.1cm}
 	 + \int_0^t \left(f_1(x) - f_2(x, t-s) \right)^2 \dd s 
		\Big) \pi (\dd x) \\ 
		&\eqqcolon \frac{b}{2} \int_V (I_1 + I_2) \pi (\dd x).
	\end{align*}
	An application of Fubini's theorem gives the following bound for $I_1$
	\begin{align*}
		t^{-1} I_1 &= t^{-1} \int_0^\infty \left(
		\int_u^{t+ u} f(x, v) \dd v
		\right)^2 \dd u \\
        & \leq t^{-1} f_1(x) \int_0^\infty \int_u^{t+u} f(x, v) \dd v \dd u \\ 
		&= f_1(x) \int_0^t f(x, v) \frac{v \wedge t}{t} \dd v, 
		\end{align*}
	which, combined with Lebesgue's dominated convergence theorem, implies 
	that $t^{-1} I_1 \to 0$.
	Similar calculation shows that $I_2 \sim t f_1(x)^2$. Therefore, $Q(t) \sim \frac{b}{2} t \int_V f_1(x)^2 \pi (\dd x )$, as $t \to \infty$. The result now follows from \cite[Theorem 1.1]{orey1972} with $v(t) = \frac{b}{2} t \int_V f_1^2 \dd \pi$.

\subsection{Proof of Proposition \ref{prop:supfou-existence}}

The claims follow by verifying the conditions \eqref{eq:RR-cond} 
(\cite[Proposition 34]{barndorff2018book}), which boils down to check the 
integrability conditions of the form 
\begin{equation} \label{eq:int-cond}
	\int_\R \int_{(0,\infty)} \psi(f(x,u)) \pi (\dd x) \dd u < \infty,
\end{equation}
for some $\psi \geq 0$. Note that if $f(x,u) = f(xu)$ then
\begin{equation*}
	\int_\R \int_{(0,\infty)} \psi(f(x,u)) \pi (\dd x) \dd u 
	= \int_{(0,\infty)} x^{-1} \pi(\dd x) \int_\R 
	\psi( f(y)) \dd y.
\end{equation*}
Therefore, \eqref{eq:int-cond} is equivalent to
\begin{equation} \label{eq:int-cond-fxy}
	\int_{(0,\infty)} x^{-1} \pi (\dd x )  < \infty 
	\quad \text{and} \quad
	\int_\R \psi( f(y)) \dd y < \infty.
\end{equation}
First, we show the integrability with respect to the large jump part.

\begin{lemma}
	Let $\kappa > 0$. Then
	\begin{equation} \label{eq:Poi-int}
		\int_{(0, \infty) \times (0, \infty) \times 
			(\R \backslash[-1,1])} 
		|z| (xu)^{\kappa - 1} e^{-xu} \mu(\dd x, \dd u, \dd z)
	\end{equation}
	exists if and only if
	\begin{equation*}
		\int_{(0,\infty)} x^{-1} \pi(\dd x) < \infty \quad \text{ and } \quad  
		\int_{|z| > 1} \log |z| \, \lambda(\dd z) < \infty.	
	\end{equation*}
\end{lemma} 

\begin{proof}
The integral in \eqref{eq:Poi-int} exists if and only if 
\begin{equation*}
\int_0^\infty \int_{(0, \infty)} \int_{|z| > 1}
\left( |z| (xu)^{\kappa - 1} e^{-xu} \wedge 1 \right) 
\lambda(\dd z)  \pi(\dd x) \dd u  < \infty,
\end{equation*}
or equivalently  by \eqref{eq:int-cond-fxy}, $m_{-1}(\pi) < \infty$ and 
\begin{equation*}
\int_0^\infty \int_{|z| > 1} (|z| y^{\kappa - 1} e^{-y} \wedge 1) 
\lambda(\dd z) \dd y < \infty.	
\end{equation*}
Fix $\kappa > 0$ and write $h(y)=y^{1-\kappa}e^y$.
Then, for $y \geq \kappa-1$, $h$ is strictly increasing, and hence its restriction to $[\kappa-1,\infty)$ is invertible.
 By Karamata's theorem \cite[Theorem 1.5.11]{bingham1989rv}, we have that
\begin{equation*}
\begin{split}
\int_{h^{\leftarrow}(|z|)}^\infty e^{-y} y^{\kappa -1} \dd y 
& = 
\int_{\exp(h^\leftarrow(|z|))}^\infty u^{-2} (\log u)^{\kappa -1} \dd u \\
& \sim \left[ u^{-1} (\log u)^{\kappa -1} \right]_{u=\exp(h^\leftarrow (|z|))}
= |z|^{-1},
\end{split}
\end{equation*}
where the asymptotic equivalence holds as $|z| \to \infty$. Therefore, for 
$K > 0$ large enough  
\begin{align*}
& \int_{|z| > K}  
\int_0^\infty  (|z| y^{\kappa - 1} e^{-y} \wedge 1)  \dd y \lambda(\dd z)  \\
& \leq \int_{|z| > K}  
\left( h^{\leftarrow}(|z|) + 
\int_{h^{\leftarrow}(|z|)}^\infty 
|z| y^{\kappa -1} e^{-y}  \dd y
\right) \lambda(\dd z) \\
& \leq \int_{|z| > K} \left( h^\leftarrow (|z|) + 2 \right) 
\lambda(\dd z).
\end{align*}
Since $h^\leftarrow (|z|) \sim \log|z|$, as $|z| \to \infty$,
the latter integral is finite if and only 
if $\int_{|z| > 1 } \log |z| \lambda(\dd z) < \infty$, as claimed. Finally, for the rest of the integral, we have that
\begin{align*}
& \int_{1 < |z| \leq K}  \int_0^\infty  
(|z| y^{\kappa - 1} e^{-y} \wedge 1) \dd y \lambda(\dd z) \\
& \leq \int_{1 < |z| \leq K} K \int_{0}^\infty y^{\kappa -1} e^{-y} 
\dd y \lambda(\dd z) < \infty.
\end{align*}
\end{proof} 

For the small jump part, we have the following result.
\begin{lemma} \label{lemma:supfOU-2}
	The compensated integral 
	\begin{equation} \label{eq:fsupOU-int}
		\int_{(0, \infty) \times (0, \infty) \times [-1,1]} 
		|z| (xu)^{\kappa - 1} e^{-x u} 
		(\mu - \nu) (\dd x, \dd u, \dd z) 
	\end{equation}
	exists if and only if $m_{-1}(\pi) < \infty$ and 
	\begin{itemize}
		\item[(i)] $\kappa > 1/2$, or
		\item[(ii)] $\kappa = 1/2$ and $\int_{|z|\leq 1} z^2 \log |z|^{-1} \lambda (\dd z) < \infty$, or
		\item[(iii)] $\kappa < 1/2$ and $\int_{|z| \leq 1 } |z| ^{\frac{1}{1-\kappa}} \lambda(\dd z) < \infty$.
	\end{itemize}
\end{lemma} 

\begin{proof}
	For $y > 0$, denote $f_\kappa(y) = y^{\kappa - 1} e^{-y}$.
	In our setup with $a = b = 0$ and $\lambda$ concentrated on $[-1,1]$ the functions $U$ and $V_0$ in \eqref{eq:UV0} have the form
	\begin{align*}
		U(y) & =  \int_{|z| \leq 1} (z y \bone( |zy| \leq 1) - z y) \lambda(\dd z)
		=  - \bone (y > 1) y \int_{|z| \in (y^{-1}, 1]} z \lambda(\dd z), \\ 
		V_0(y) & = \int_{|z| \leq 1} (1 \wedge y^2 z^2) \lambda(\dd z).	
	\end{align*}
	Conditions \eqref{eq:RR-cond} combined with \eqref{eq:int-cond-fxy} imply that the 
	compensated integral in \eqref{eq:fsupOU-int} exists if and only if $m_{-1}(\pi) < \infty$ and 
	\begin{equation} \label{eq:fsupOU-RR}
		\int_0^\infty |U(f_\kappa(y)) | \dd y < \infty, \quad 
		\int_0^\infty V_0(f_\kappa(y)) \dd y < \infty.
	\end{equation}
	
	First, we check that the first integral in \eqref{eq:fsupOU-RR} is finite under our assumptions.
	If $f_\kappa(y) \leq K$ for some $K > 0$, then 
	\begin{align*}
		\int_0^\infty |U(f_\kappa(y))| \dd y  &= 
		\int_0^\infty \bone(f_\kappa (y) > 1) f_\kappa(y)  
		\left| \int_{|z| \in (f_\kappa(y)^{-1}, 1]} z \lambda(\dd z) \right| \\
		& \leq K \int_{|z| \in (1/K, 1]} |z| \lambda(\dd z) \Leb ( f_\kappa(y) > 1) < \infty.
	\end{align*}
	Therefore, by \eqref{eq:int-cond-fxy}, the integral exists whenever $\kappa \geq 1$.
	
	For $\kappa < 1$, we have that
	\begin{equation} \label{eq:fsupOU-aux1}
		\int_0^\infty |U(f_\kappa(y))| \dd y  
		\leq \int_{|z| \leq 1} |z|  \lambda(\dd z) \int_0^\infty f_\kappa(y) \bone ( f_\kappa(y) > |z|^{-1}) \dd y.
	\end{equation}
	Since $f_\kappa(y) \sim y^{\kappa - 1}$ as $y \downarrow 0$, as $x \to \infty$
	\begin{equation*}
		\int_{0}^\infty f_\kappa(y) \bone ( f_\kappa (y) > x) \dd y \sim \int_{0}^{x^{-1/(1-\kappa)}} y^{\kappa -1} \dd y \sim 
		\kappa^{-1} x^{-\frac{\kappa}{1-\kappa}}.  
	\end{equation*}
	Therefore, by \eqref{eq:fsupOU-aux1}
	\begin{equation*}
		\int_0^\infty |U(f_\kappa(y))| \dd y \leq c \int_{|z| \leq 1} |z|^{\frac{1}{1-\kappa}} \lambda(\dd z) < \infty,    
	\end{equation*}
	where for $\kappa \geq 1/2$ the latter integral is finite, since $\lambda$ is a L\'evy measure,
	while for $\kappa < 1/2$ the finiteness follows from the assumption.
	
	Next, consider the second integral in \eqref{eq:fsupOU-RR}. 
	We have $\int_0^\infty f_\kappa (y)^2 \dd y < \infty$ for $\kappa > 1/2$. Thus, assume $\kappa \leq 1/2$. We have
	\begin{equation*} \label{eq:fsupOU-aux2}
		\begin{split}		
			\int_0^\infty V_0(f_\kappa(y)) \dd y 
			& = 
			\int \bone (f_\kappa(y) \leq 1) \int_{|z| \leq 1} z^2  
			f_\kappa(y)^2 \lambda(\dd z) \dd y \\
			& \quad + \int \bone (f_\kappa(y) > 1) \int_{|z| \leq 1} (1 \wedge f_\kappa(y)^2 z^2) \lambda(\dd z) \dd y.
		\end{split}
	\end{equation*}
	The first integral is finite since $\int_0^\infty f_\kappa(y) \dd y < \infty$. For the 
	second term, by the asymptotic relation $f_\kappa(y) \sim y^{\kappa -1}$ as $y \downarrow 0$, 
	we have as $x \to \infty$
	\begin{align*}
		\int_0^\infty f_\kappa(y)^2 \bone(f_\kappa(y) < x) \dd y 
		& \sim \int_{x^{-\frac{1}{1-\kappa}}}^\infty f_\kappa(y)^2 \dd y \\
        & \sim \int_{x^{-\frac{1}{1-\kappa}}}^1 y^{2\kappa - 2} \dd y \\
        & \sim 
		\begin{cases} 
			\frac{1}{1-2 \kappa} x^{-\frac{1- 2\kappa}{1 - \kappa}}, & \kappa < 1/2, \\
			2 \log x, & \kappa = 1/2.
		\end{cases}
	\end{align*}
	Therefore, for $\kappa < 1/2$ we have  
	\begin{align*}
		& \int \bone (f_\kappa(y) > 1) \int_{|z| \leq 1} (1 \wedge f_\kappa(y)^2 z^2) 
		\dd y \lambda(\dd z) \\
		& = \int_{|z| \leq 1} 
		\left( \Leb(f_\kappa^{-1}(|z|^{-1}, \infty)) + \int f_\kappa(y)^2 z^2 \bone(f_\kappa(y) |z| \leq 1) \dd y
		\right) \lambda(\dd z)  \\
		& = c \int_{|z| \leq 1} \left( |z|^{\frac{1}{1-\kappa}} + 
		z^2 |z|^{\frac{2\kappa - 1}{1-\kappa}} \right) \lambda(\dd z) \\
		&  = c \int_{|z| \leq 1} |z|^{ \frac{1}{1-\kappa}} \lambda(\dd z) < \infty.
	\end{align*}
	For $\kappa = 1/2$ an additional $\log |z|^{-1}$ factor appears above. Summarizing, if the assumptions on $\pi$ and $\lambda$
	are satisfied, then the process exists. The necessity of the assumptions follows from the last calculations.
\end{proof}

\noindent
\textbf{Acknowledgements.} 
The authors would like to thank the Associate Editor and the referees 
for their careful reading and for the comments that led to a considerable 
improvement of the paper.

Danijel Grahovac was supported by the Croatian Science Foundation under 
the project Scaling in Stochastic Models (HRZZ-IP-2022-10-8081). 
Dominik Mihal\v{c}i\'c was partially supported by the project Stochastic 
Models and Applications (581-UNIOS-56), funded by the 
European Union -- NextGenerationEU through the National Recovery and 
Resilience Plan 2021--2026 of the Republic of Croatia.

\end{document}